\documentclass[a4paper,11pt]{article}
\usepackage{array}
\usepackage{theorem}
\usepackage{amsmath,amscd,amssymb, color}
\usepackage{latexsym}
\usepackage{epsfig}
\usepackage{xypic}
\usepackage{url}
\usepackage{relsize}
\theorembodyfont{\sl}

\newtheorem{lemma}{Lemma}[section]

\newtheorem{theorem}[lemma]{Theorem}

\newtheorem{definition}[lemma]{Definition}
\newtheorem{remark}[lemma]{Remark}
\newtheorem{question}[lemma]{Question}

\newcommand{\FF}{\mathbb F}

\newcommand{\NN}{\mathbb N}

\newcommand{\PP}{\mathbb P}
\newcommand{\QQ}{\mathbb Q}
\newcommand{\RR}{\mathbb R}

\newcommand{\ZZ}{\mathbb Z}

\renewcommand{\cD}{\mathcal D}
\newcommand{\cE}{\mathcal E}

\newcommand{\cM}{\mathcal M}

\newcommand{\GL}{\mathop{\mathrm {GL}}\nolimits}

\newcommand{\SL}{\mathop{\mathrm {SL}}\nolimits}

\newcommand{\Orth}{\mathop{\null\mathrm {O}}\nolimits}

\newcommand{\Ap}{\mathop{\mathrm {Ap}}\nolimits}

\newcommand{\Stab}{\mathop{\mathrm {Stab}}\nolimits}

\renewcommand{\Im}{\mathop{\mathrm {Im}}\nolimits}

\newcommand{\NS}{\mathop{\null\mathrm {NS}}\nolimits}
\newcommand{\Num}{\mathop{\null\mathrm {Num}}\nolimits}

\newcommand{\Eig}{\mathop{\null\mathrm {Eig}}\nolimits}

\newcommand*{\shifttext}[2]{
  \settowidth{\@tempdima}{#2}
  \makebox[\@tempdima]{\hspace*{#1}#2}
}
\newcommand{\shiftleft}[2]{\makebox[0pt][r]{\makebox[#1][l]{#2}}}

\newcommand{\Isperp}{Is^{\shiftleft{3pt}{\raisebox{2pt}{$\mathsmaller{\mathsmaller{\perp}}$}}}\hspace{4pt}}
\newcommand{\IsperpB}{Is^{\shiftleft{3pt}{\raisebox{2pt}{$\mathsmaller{\mathsmaller{\perp}}$}}}}

\newcommand{\Kthree}{\mathop{\mathrm {K3}}\nolimits}
\newcommand{\En}{\mathop{\mathrm {En}}\nolimits}

\newcommand{\qedsymbol}{\mbox{$\Box$}}
\newcommand{\qed}{\unskip\nobreak\hfil\penalty50\hskip1em\hbox{}\nobreak
\hfill\qedsymbol\parfillskip=0pt\finalhyphendemerits=0}
\newenvironment{proof}{\begin{ProofwCaption}{Proof}}{\end{ProofwCaption}}
\newenvironment{ProofwCaption}[1]
 {\addvspace\theorempreskipamount \noindent{\it #1.}\rm}
 {\qed \par \addvspace\theorempostskipamount}

\newcommand\blfootnote[1]{
  \begingroup
  \renewcommand\thefootnote{}\footnote{#1}
  \addtocounter{footnote}{-1}
  \endgroup
}

\begin{document}

\title{Moduli of polarized Enriques surfaces -- computational aspects}
\author{Mathieu Dutour Sikiri\'c and Klaus Hulek}
\date{}
\maketitle

\begin{abstract}
Moduli spaces of (polarized) Enriques surfaces can be described as open subsets of modular varieties of orthogonal type. It was shown by Gritsenko and Hulek that there are, up to isomorphism, only finitely many
different moduli spaces of polarized Enriques surfaces. Here we investigate the possible arithmetic groups and show that there are exactly $87$ such groups up to conjugacy. We also show that all moduli spaces
are dominated by a moduli space of polarized Enriques surfaces of degree $1240$.
Ciliberto, Dedieu, Galati and Knutsen have also investigated moduli spaces of polarized Enriques surfaces in detail. We discuss how our enumeration relates to theirs.
We further compute the Tits building of the groups in question.
Our computation is based on groups and indefinite quadratic forms and the algorithms used are explained.
\end{abstract}

\section{Introduction}\label{sec:intro}
\blfootnote{
2020 \emph{Mathematics Subject Classification}. 14J10, 14J28, 11H55, 20B40.
}
The moduli space $\cM^0_{\En}$ of Enriques surfaces is an open subset of a $10$-dimensional orthogonal
modular variety, which was shown by Kond\=o \cite{Kon1} to be rational.
This description is obtained by considering the universal cover of Enriques surfaces, which leads to the moduli space of $K3$ surfaces with a fixed-point free involution. Indeed, $\cM^0_{\En}$
can be viewed as the moduli space of $N$-polarized $K3$ surfaces where
\begin{equation}\label{def:latticeN}
N = U + U(2) + E_8(-2).
\end{equation}\label{def:N}
Here $U$ denotes a hyperbolic plane, $E_8(-1)$ is the negative-definite $E_8$-lattice and
$U(2)$ and $E_8(-2)$ means that the bilinear forms are multiplied by $2$.
These $K3$ surfaces carry a non-symplectic free involution giving rise to a quotient which is an Enriques surface.

Taking a slightly different viewpoint, one can also consider moduli spaces of {\em polarized } Enriques surfaces, i.e. Enriques surfaces with an ample line bundle. These moduli spaces come in two flavours,
namely as moduli spaces of {\em polarized } or {\em numerically polarized } Enriques surfaces, depending on whether one considers the polarization as an element in the N\'eron-Severi group $\NS(S)$
or the group $\Num(S)$ of divisors modulo numerical equivalence. We recall that $\Num(S)$ is the quotient of $\NS(S)$ by the $2$-torsion element given by the canonical class $K_S$.
It was shown in \cite{GH3} that the moduli spaces $\cM^a_{\En,h}$ of numerically polarized Enriques surfaces are open subsets of $10$-dimensional orthogonal modular varieties $\cM_{\En,h}$
(see (\ref{equ:descriptionmoduliEnriques})) and the moduli spaces
 $\widehat\cM_{\En,h}^{a}$ of polarized Enriques surfaces are \'etale $2:1$ covers $\widehat\cM_{\En,h}^{a} \to \cM^a_{\En,h}$.
 In \cite{GH3} we also asked the question when this covering is connected. A complete answer
 was given in \cite[Theorem 1.1]{Kn}: the space  $\cM_{\En,h}^{a}$ is connected if and only if the class $h$ is not $2$-divisible in $\Num(S)$.

 Moduli spaces of polarized Enriques surfaces behave in some ways very differently from moduli spaces of polarized $K3$ surfaces.
 Indeed, it was shown in \cite[Theorem 1.1]{GH3} that there  are only finitely many moduli spaces, up to isomorphism, of (numerically) polarized Enriques surfaces. The starting
 point of this paper is the question: how many different moduli spaces of Enriques surfaces exist? Here we shall treat this question from the point of view of orthogonal modular varieties.

 To describe the results of this paper we need some more details concerning
 moduli spaces of numerically polarized Enriques surfaces, which are all open subsets of orthogonal modular varieties.
As usual (see also Section \ref{sec:construction} for more details), we denote by  $\cD_N$ a connected component of the $10$-dimensional type IV domain $\Omega_N$ associated to $N$ and by
$\Orth(N)$ and $\Orth^+(N)$ the orthogonal group and the orthogonal group of transformations with real spinor norm $1$. These act on $\Omega_N$  and $\cD_N$ respectively  and we set
\begin{equation*}
\cM_{\En}:= \Orth^+(N) \backslash \cD_N.
\end{equation*}
The moduli space $\cM^0_{\En}$ of  Enriques surfaces is the open subset of $\cM_{\En}$
\begin{equation*}
\cM^0_{\En}:= \cM_{\En} \setminus \Delta_{-2}
\end{equation*}
where $\Delta_{-2}$ is the image of all hyperplanes orthogonal to roots $r$ in $N$.
This is necessary to ensure that we really have period points on Enriques surfaces.
By \cite[Theorem 2.13]{Nam} the hypersurface $\Delta_{-2}$ is irreducible.

There is also the notion of moduli spaces of Enriques surfaces with a level-$2$ structure.
For this we consider the dual lattice of $N$, which we denote by $N^{\vee}$, and the stable orthogonal group $\widetilde \Orth(N)$, which is defined as the group of
all elements in $\Orth(N)$ acting trivially on the discriminant $D(N)=N^{\vee}/N$. We set
\begin{equation*}
\widetilde \Orth^+(N):= \Orth^+(N)  \cap \widetilde \Orth(N)
\end{equation*}
and note that this is an index $2$ subgroup since the reflection with respect to a vector of length $2$ in the
summand $U$ of $N$ gives an element in $\widetilde \Orth(N)$ with real spinor norm $-1$.
Let
\begin{equation*}
\widetilde {\cM}_{\En}:= \widetilde \Orth^+(N)\backslash \cD_N.
\end{equation*}
The open subset
\begin{equation*}
\widetilde {\cM}^0_{\En}:= \widetilde {\cM}_{\En} \setminus \widetilde {\Delta}_{-2}
\end{equation*}
defined as the complement of the hypersurfaces orthogonal to the roots can be interpreted as the {\em moduli space of Enriques surfaces with a level 2 structure}.

We recall that $D(N)=N^{\vee}/N\cong (\FF_2)^{10}$ and
\begin{equation*}
\Orth(D(N)) \cong \Orth^+(\FF_2^{10})
\end{equation*}
is the orthogonal group of {\em even} type whose order is $\left\vert\Orth^+(\FF_2^{10})\right\vert = 2^{21} \cdot 3^5 \cdot 5^2 \cdot 7 \cdot 17 \cdot 31$. For details see \cite[\S 1]{Kon2}, \cite[Chap. I, \S 16, Chap. II. \S 10]{Die}.
We also recall that $\Orth(N) \to \Orth(D(N))$ is surjective (see \cite[Theorem 3.6.3]{Nik} and
\begin{equation}\label{equ:quotientgroups}
\Orth(D(N)) \cong \Orth(N)/\widetilde\Orth(N) \cong \Orth^+(N)/\widetilde\Orth^+(N).
\end{equation}

We will describe the construction of moduli spaces of polarized Enriques surface in more detail in Section \ref{sec:construction}. Here we only want to state
that all moduli spaces $\cM^a_{\En,h}$ are open subsets of orthogonal modular varieties
\begin{equation} \label{equ:descriptionmoduliEnriques}
{\cM}_{\En,h}:= \Gamma^+_h \backslash \cD_N
\end{equation}
where
\begin{equation} \label{equ:subgroupsconatined}
\widetilde\Orth^+(N) \subset  \Gamma^+_h \subset \Orth^+(N).
\end{equation}
From this one has to remove the hyperplanes orthogonal to the roots and some hyperplanes which are orthogonal
to certain $(-4)$-vectors. The latter is necessary to ensure that $h$ represents an ample class, not removing these hyperplanes means that we are also considering quasi-polarizations, i.e. nef and big
line bundles. Here we are exclusively concerned with the orthogonal varieties ${\cM}_{\En,h}$. Obviously, there are only finitely many possible choices of subgroups $\Gamma_h$. Each such choice
defines an orthogonal modular variety ${\cM}_{\En,h}$ which is covered by $\widetilde {\cM}_{\En}$ and covers ${\cM}_{\En}$ in turn:
\begin{equation*}
\widetilde {\cM}_{\En} \to \cM_{\En,h} \to {\cM}_{\En}.
\end{equation*}
Note that the maps involved here are not necessarily Galois coverings.

In the situation described here, a number of natural questions arise which we want to address in this paper.
The first question is to ask for the number of possible modular varieties which arise in connection with moduli spaces of polarized Enriques surfaces. We rephrase this question in terms
of arithmetic groups:

\begin{question}\label{qu:1}
How many subgroups $\Gamma^+_h$, arising from moduli spaces ${\cM}_{\En,h}$ of polarized Enriques surfaces, (cf. (\ref{equ:descriptionmoduliEnriques})), exist (up to conjugacy)?
\end{question}

We will see in Theorem \ref{teo:87groups} that there are $87$ such conjugacy classes. In Tables \ref{table_subgroups1} and \ref{table_subgroups2} we will provide further information
about these groups, in particular their index in $\Orth^+(N)$ (which is
equivalent to knowing the index of $\tilde \Orth^+(N)$ in $\Gamma^+_h$).
This can be rephrased in terms of subgroups of the finite group orthogonal group $\Orth^+(\FF_2^{10})$, see Question \ref{qu_subgroup_enum}.

The next question concerns the relation between the degree $h^2=2d$ of a polarization and the possible groups $\Gamma^+_h$. In the case of $K3$ surface, given the degree $h^2=2d$ of
a primitive polarization, we obtain an irreducible moduli space of $2d$-polarized $K3$ surfaces. The reason is that the $K3$-lattice $L_{K3}=3U + 2E_8(-1)$ is unimodular and
the group $\Orth(L_{K3})$ acts transitively on all primitive vectors of fixed norm. This is no longer true in the case of Enriques surfaces.
Indeed, for given degree $h^2=2d> 2$ one has to expect many primitive vectors $h$ which are not equivalent modulo the action of
the isometry group of the N\'eron-Severi lattice $M(1/2):=U+E_8(-1)$.

This leads us to our next
\begin{question}\label{qu:3}
Enumerate all inequivalent primitive vectors $h \in U+E_8(-1)$ of given (small) degree $h^2=2d$ under the action of the group $\Orth(U+E_8(-1))$.
\end{question}

We will give an answer to this in Theorem \ref{teo:smalldegreevectors}. In Table \ref{ListVectorsNormAtMost30} we gather the information as to which polarizations define conjugate groups $\Gamma^+_h$.

In \cite[Proposition 5.7]{GH3} the existence of a polarization $h_0$ was shown such that $\Gamma^+_{h_0}$ is minimal, i.e. $\Gamma^+_{h_0}=\widetilde\Orth^+(N)$.
This is of interest as the corresponding modular variety ${\cM}_{\En,h_{0}}=\widetilde {\cM}_{\En}$ covers all varieties ${\cM}_{\En,h}$. Hence it is natural to ask
\begin{question}\label{qu:2}
What is the minimal degree $d_{\min}=h_0^2$ such that there exists a vector $h_0$ with $\Gamma^+_{h_0}=\widetilde\Orth^+(N)$, i.e. ${\cM}_{\En,h_{0}}=\widetilde {\cM}_{\En}$?
Is such a vector of minimal degree unique?
\end{question}

We shall provide an answer to this question in Theorem \ref{teo:degreemaximalgroup} where we will see that there is a unique such vector $h_0$ of minimal degree $h_0^2=1240$.

Naturally, moduli spaces of polarized Enriques surfaces of small degree have been studied classically. For a discussion of classical constructions for $d \leq 10$ we refer to Dolgachev's paper
\cite{Dol3}.  In the case of degree $4$ Casnati \cite{Cas} studied polarizations which are base-point free and lead to a $4:1$ cover of the projective plane (also called Cossec-Verra polarizations).
He showed that this defines an irreducible moduli space which is also rational. There are also degree $4$ polarizations (ample line bundles) which are not base point free. These are sometimes not considered  to be
polarizations in the literature (see \cite[Section 1]{Cas}). The case of (base point free) polarizations of degree $6$ is the classical case representing Enriques surfaces as singular sextic surfaces in $\PP^3$.
For degree $10$ there exists one polarization with generically very ample line bundle.
This leads to Reye congruences, respectively degree $10$ models in $\PP^5$.
We shall discuss these cases and the relation with our calculations more systematically in Section \ref{subsec:classical}.

Ciliberto, Dedieu, Galati and Knutsen undertook a very systematic enumeration of moduli spaces of polarized Enriques surfaces in \cite{CDGK}, based on the $\phi$-invariant of a polarization. This is the minimal
degree of a polarization on an effective elliptic curve.
This enumeration was taken further in \cite{Kn} where it was shown that the moduli spaces depend on a finer invariant, called the $\phi$-vector, which is the minimal (defined in a proper way) degree of the
polarization with respect to a whole isotropic $10$-sequence (see \cite[Theorem 1.4]{Kn})
This leads us to the

\begin{question}\label{qu:4}
How can the enumerations given by our methods and that of Ciliberto et al. be matched?
\end{question}

A complete matching will be provided in Tables \ref{Expression_Isotropic_Vectors_Part1} and \ref{Expression_Isotropic_Vectors_Part2}.

When one wants to study the geometry of moduli spaces one typically has to work with projective compactifications of the modular varieties $\Gamma^+_h \backslash \cD_N$.
The natural choices here are the Baily-Borel and
toroidal compactifications. The first is canonically defined for all orthogonal modular varieties, the second involves a choice of fans. In either case it is important to know the cusps
as these are in $1:1$ correspondence with the boundary components of the Baily-Borel compactification. In the orthogonal case one has $0$-dimensional cusps (points) and
$1$-dimensional cusps (modular curves) which have to be added to the orthogonal modular variety to obtain the Baily-Borel compactification.
We recall that  for all arithmetic orthogonal groups $\Gamma$ of lattices of signature $(2,n)$ the $0$ and $1$-dimensional
cusps are in $1:1$ correspondence with the $\Gamma$-orbits of rational isotropic lines and rational isotropic planes respectively.
More generally, a $0$-dimensional cusp is contained in the closure of a $1$-dimensional cusp if and only if $l \subset e$ for some representatives of the corresponding isotropic line and plane respectively.
The Tits building is the $1$-complex whose vertices are the orbits of isotropic lines and planes and whose edges are given by the inclusion relation. The Tits building $\mathcal T (\Gamma^+_h)$ encodes the
combinatorial structure of the boundary of the Baily-Borel compactification of  $\Gamma^+_h \backslash \cD_N$.
This leads to the
\begin{question}\label{qu:5}
How many $0$- and $1$-dimensional cusps do the varieties $\Gamma^+_h \backslash \cD_N$ have. More generally, what can we say about the Tits building $\mathcal T (\Gamma^+_h)$?
\end{question}

This question will be addressed in Section \ref{subsec:isotropic}.

\subsection* {Acknowledgements} The first author is grateful to DFG for partial support under DFG Hu 337/7-2 and to Leibniz University Hannover for hospitality. He also thanks G.~Nebe and S.~Brandhorst
for an exchange of ideas at an early stage of then project.
We would like to thank A.~L.~Knutsen for interesting discussions concerning \cite{CDGK} and him and S.~Brandhorst for very helpful comments on a first version of this manuscript.

\section{Construction of the moduli spaces}\label{sec:construction}

In this section we want to give more details on the construction of the moduli spaces and the groups involved. The starting point is the fact that for an Enriques surface
$S$ the group of divisors modulo numerical equivalence is
\begin{equation*}
H^2(S,\ZZ)_f=\Num (S) \cong U + E_8(-1).
\end{equation*}
The fact that the canonical class $K_S$ is $2$-torsion implies the existence of an
\'etale $2:1$ cover $p: X \to S$ where  $X$ is a $\Kthree$ surface. We denote the corresponding involution on $X$ by $\sigma: X \to X$.
It is well known that the intersection form equips $H^2(X,\ZZ)$ with the structure of a lattice, namely
$$
H^2(X,\ZZ)\cong 3U + 2E_8(-1) =: L_{\Kthree}
$$
where we refer to $L_{\Kthree}$ as the $\Kthree$ lattice.
Under the $2:1$ cover $p: X \to S$ the intersection form is multiplied by a factor $2$
and thus
$$
p^*(H^2(S,\ZZ)) \cong U(2) + E_8(-2)=:M.
$$
By \cite[Theorem 1.14.4]{Nik} the primitive embedding of the lattice  $U(2) + E_8(-2)$ into the $\Kthree$ lattice $L_{\Kthree}$ is unique (up to the action of $\Orth(L_{K3})$.
Hence we may assume that
$M$ is embedded into $L_{\Kthree}$ by the embedding $(x,u) \mapsto (x,0,x,u,u)$ where $x\in U(2), u \in E_8(-2)$.
When we refer to the sublattice $M$ of $L_{\Kthree}$ we will always assume this embedding.
The sublattice $M$ also has an interpretation in terms of the involution
$$
\rho: L_{\Kthree}=3U + 2E_8(-1) \to L_{\Kthree}=3U + 2E_8(-1),
$$
$$
\rho(x,y,z,u,v) = (z,-y,x,v,u).
$$
Clearly $M$ can be identified with the $(+1)$-eigenspace  $\Eig(\rho)^+$ of $\rho$.
The  $(-1)$-eigenspace $\Eig(\rho)^-$ can be identified with the lattice $N$ as defined in (\ref{def:N}). Indeed, we can embed the lattice $N$ primitively into $L_{K3}$ by
$(y,z,v) \mapsto (z,y,-z,v,-v)$ and this gives
\begin{equation*}
\Eig(\rho)^-=M^{\perp}_{L_{\Kthree}}\cong N.
\end{equation*}

We shall now explain how the groups $\Gamma^+_h$ arise.
When one wants to construct moduli spaces of numerically polarized Enriques surfaces one considers pairs $(S,h)$ where $h$ is the class of a numerical polarization.
This defines an element
$h\in U + E_8(-1) = M(1/2)$ of positive degree  $h^2=2d>0$. For what follows we can and will assume that this vector is primitive.
Given $h$ one has to consider the stabilizer
\begin{equation*}
\Orth(M(1/2),h)=\Orth(M,h)= \{g \in \Orth(M(1/2))=\Orth(M) \mid g(h)=h \}.
\end{equation*}
The natural maps $\pi_M: \Orth(M) \to \Orth(D(M))$ and $\pi_N: \Orth(N) \to \Orth(D(N))$ are surjective. Since $M$ and $N$ are orthogonal to each other in the $K3$ lattice
$L_{K3}$, the discriminant groups $D(M)$ and  $D(N)$ are naturally isomorphic:
\begin{equation*}
(D(M), q_M) \cong (D(N), -q_N).
\end{equation*}
Here $q_M$ and $q_N$ are the induced quadratic forms.
We shall forthwith identify these discriminant groups and hence also $\Orth(D(M))$ and $\Orth(D(N))$.

The crucial definition is
\begin{equation}\label{equ:modulargroup}
\Gamma_h: =\pi_N^{-1}(\pi_M(\Orth(M,h))) \subset \Orth(N).
\end{equation}
Since $\widetilde \Orth(N) \subset \Gamma_h$ is a normal subgroup of $\Orth(N)$ of finite index, the group $\Gamma_h$ is
an arithmetic subgroup of $\Orth(N)$. We again note that the subgroup
\begin{equation*}
\Gamma^+_h = \Gamma_h\cap \Orth^+(N)
\end{equation*}
has index $2$.

In fact we can rephrase our Question \ref{qu:1} on the groups $\Gamma^+_h$ entirely in terms of subgroups of $\Orth^+(\FF_2^{10})$.  For this let
\begin{equation*}
\bar \Gamma_h:= \pi_M(\Orth(M,h) \subset \Orth^+(\FF^{10}_2).
\end{equation*}

Since the natural map $\Orth^+(N) \to \Orth(D(N)) \cong \Orth^+(\FF_2^{10})$
is surjective Question \ref{qu:1} can be solved by giving an answer to
\begin{question}\label{qu_subgroup_enum}
How many subgroups (up to conjugacy) of the form $\bar \Gamma_h$ are there in $\Orth^+(\FF_2^{10})$?
\end{question}

\section{The computations}\label{sec:proof}
\subsection{Some basic facts and roots}\label{subsec:roots}

The lattice $M(-1/2) = U + E_8$ is known under different names. It is actually the root lattice of the hyperbolic Coxeter group $E_{10}$ (with $E_n$, $n\leq 8$ being the classical ones and $E_9$ being the affine extension of $E_8$). It is also the even Lorentzian lattice $II_{9,1}$. Another common name is $E_8^{++}$, see \cite{KMW} for more details.
We will use the following Gram matrix for $M(1/2)$:
\begin{equation}\label{equ:G}
G = \left(\begin{array}{cccccccccc}
0 & 1 & 0 & 0 & 0 & 0 & 0 & 0 & 0 & 0\\
1 & 0 & 0 & 0 & 0 & 0 & 0 & 0 & 0 & 0\\
0 & 0 & -2 & 0 & 0 & 0 & 0 & 0 & -1 & 0\\
0 & 0 & 0 & -2 & 1 & 0 & 1 & -1 & 0 & -1\\
0 & 0 & 0 & 1 & -2 & -1 & 0 & 0 & 0 & 1\\
0 & 0 & 0 & 0 & -1 & -2 & 0 & 0 & -1 & 1\\
0 & 0 & 0 & 1 & 0 & 0 & -2 & 1 & -1 & 1\\
0 & 0 & 0 & -1 & 0 & 0 & 1 & -2 & 0 & 0\\
0 & 0 & -1 & 0 & 0 & -1 & -1 & 0 & -2 & 1\\
0 & 0 & 0 & -1 & 1 & 1 & 1 & 0 & 1 & -2
\end{array}\right)
\end{equation}
We shall first collect some basic facts about this lattice.
It is known to be $2$-reflective, see \cite{Nik2}.
Hence Vinberg's algorithm \cite{Vin} can be employed to compute a fundamental domain $D$ of the Weyl group of the lattice.
Here we give a list of roots that define the facets of (a possible choice of) $D$, that is a list of simple roots.
The roots $r$ satisfy $r^2 = -2$, are numbered from $-1$ to $8$ and have the coordinates:
\begin{equation*}
\begin{array}{rcl}
-1 &=& (-1, 1, 0, 0, 0, 0, 0, 0, 0, 0)\\
0 &=& (0, -1, -1, 1, 1, -2, -2, -2, 2, -1)\\
1 &=& (0, 0, 0, 0, -1, 1, 1, 1, -1, 0)\\
2 &=& (0, 0, 0, 0, 1, -1, 0, 0, 0, 0)\\
3 &=& (0, 0, 0, 0, 0, 0, -1, 0, 1, 0)\\
4 &=& (0, 0, 1, 0, -1, 1, 1, 0, -2, -1)\\
5 &=& (0, 0, -1, -1, 0, -1, -1, 0, 2, 1)\\
6 &=& (0, 0, 0, 0, 0, 1, 0, 0, -1, 0)\\
7 &=& (0, 0, 1, 1, 0, 0, 1, 0, -1, 0)\\
8 &=& (0, 0, 0, 1, 1, 0, 1, 0, 0, 0).\\
\end{array}
\end{equation*}
The associated Coxeter-Dynkin diagram is
\begin{center}
\resizebox{6cm}{!}{\includegraphics{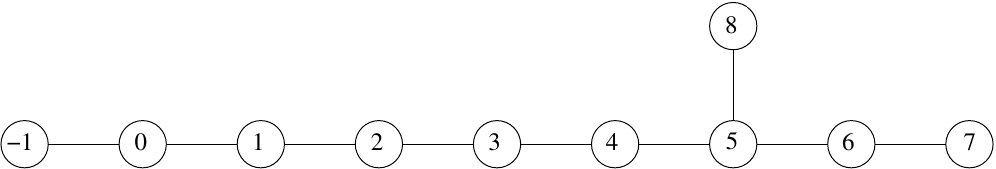}}\par
\end{center}

Since there are $10$ simple roots, it follows that the fundamental domain is simplicial.
The generators $g_i$ of the extreme rays are the following:
\begin{equation*}
\left(\begin{array}{cccccccccc}
0 & -1 & 0 & 0 & 0 & 0 & 0 & 0 & 0 & 0\\
-1 & -1 & 0 & 0 & 0 & 0 & 0 & 0 & 0 & 0\\
-2 & -2 & -2 & 1 & 1 & -3 & -3 & -2 & 3 & -2\\
-2 & -2 & -1 & 1 & 1 & -2 & -2 & -2 & 2 & -1\\
-3 & -3 & -2 & 2 & 1 & -4 & -4 & -3 & 3 & -3\\
-3 & -3 & -2 & 2 & 1 & -3 & -3 & -3 & 3 & -2\\
-4 & -4 & -3 & 3 & 2 & -6 & -5 & -4 & 5 & -4\\
-4 & -4 & -3 & 3 & 2 & -5 & -4 & -4 & 4 & -3\\
-5 & -5 & -4 & 4 & 3 & -7 & -6 & -5 & 6 & -4\\
-6 & -6 & -4 & 5 & 3 & -8 & -7 & -6 & 6 & -6
\end{array}\right)
\end{equation*}
A straightforward computation shows that these generators even define a $\ZZ$-basis of the lattice $U+E_8(-1)$ and hence the fundamental domain is in fact a basic cone.

The symmetrix matrix $W = g_i\cdot g_j$ with respect to the simple roots $g_i$, is easily computed to be
\begin{equation}\label{equ:W}
W =
\left(\begin{array}{cccccccccc}
0 & 1 & 2 & 2 & 3 & 3 & 4 & 4 & 5 & 6\\
1 & 2 & 4 & 4 & 6 & 6 & 8 & 8 & 10 & 12\\
2 & 4 & 4 & 6 & 7 & 8 & 9 & 10 & 12 & 14\\
2 & 4 & 6 & 6 & 9 & 9 & 12 & 12 & 15 & 18\\
3 & 6 & 7 & 9 & 10 & 12 & 14 & 15 & 18 & 21\\
3 & 6 & 8 & 9 & 12 & 12 & 16 & 16 & 20 & 24\\
4 & 8 & 9 & 12 & 14 & 16 & 18 & 20 & 24 & 28\\
4 & 8 & 10 & 12 & 15 & 16 & 20 & 20 & 25 & 30\\
5 & 10 & 12 & 15 & 18 & 20 & 24 & 25 & 30 & 36\\
6 & 12 & 14 & 18 & 21 & 24 & 28 & 30 & 36 & 42
\end{array}\right)
\end{equation}
From the above presentation one can also conclude that the Coxeter-Dynkin $E_{10}$ has only trivial symmetries.
It thus follows that the isometry group and the Coxeter group coincide:
\begin{equation}\label{eq:nosymmegtries}
\Orth(U + E_8(-1)) = W(U + E_8(-1)).
\end{equation}
We also mention that the Coxeter-Dynkin is simply laced, i.e. has no multiple edges (but we will not make use of this fact).

This information already allows us to give an answer to Question \ref{qu:2}:

\begin{theorem}\label{teo:degreemaximalgroup}
The minimal norm of integer vectors with trivial stabilizer in $U + E_8(-1)$ is $1240$ and in this degree there is a unique such vector.
\end{theorem}
\begin{proof}
Since the Coxeter-Dynkin diagram has no symmetries, we have already concluded in (\ref{eq:nosymmegtries}) that the isometry group and
the Coxeter group of the lattice $U+E_8(-1)$ coincide. Hence a vector $h$ has trivial stabilizer if and only if it is in the interior
of the fundamental domain.
Since the $g_i$ form a $\ZZ$-basis of the lattice, it follows that
\begin{equation*}
v = \sum_{i=1}^{10} a_i g_i \, \mbox{for some } a_i \in \NN_{>0}.
\end{equation*}
It then follows from the form of $W$ in (\ref{equ:W}), notably the observation that all entries in the matrix $W$ are non-negative and only one entry is $0$, that the minimum value for $v^2$ is obtained if and only if all $a_i=1$.  We can then conclude, again from (\ref{equ:W}), that this minimum value is $v^2=1240$.
\end{proof}

\begin{remark}
The vector $h$ with norm $h^2=1240$ is characterized by the property that $(h,r)=1$ for every root $r$ defining a wall of the Weyl chamber. This is called the
{\em Weyl vector} in \cite[Chapter 27, \S2, Theorem 1]{CS}.
\end{remark}

We note that this fits very well with the results obtained by Knutsen in \cite[Proposition 1.5]{Kn} where a geometric construction of a divisor class $h_0$ was given such that $h_0^2$ is minimal and the
corresponding moduli space
dominates all moduli spaces of numerically polarized Enriques surfaces. The divisor found by Knutsen also satisfies $h_0^2=1240$ and we checked by computer that his polarization and the polarization $h$
from Theorem \ref{teo:degreemaximalgroup} are equivalent confirming that the corresponding modular varieties are the same.

\subsection{Enumerating polarizations of small degree}\label{subsec:enumspaces1}

Our next aim is to enumerate the number of inequivalent polarizations in a given degree (for small values of $d$).

\begin{theorem}\label{teo:smalldegreevectors}
The list of non-isotropic vectors of norm at most $30$ in the fundamental domain is given in Table \ref{ListVectorsNormAtMost30}.
\end{theorem}
\begin{proof}
The matrix of scalar products $(g_i\cdot g_j)$ is positive except for the isotropic vector. We can enumerate the vectors $w$ of the form
\begin{equation*}
\sum_{k=2}^{10} \alpha_k g_k\, \mbox{ for } \alpha_k \in \NN
\end{equation*}
with $w\cdot w \leq 30$. Since $g_k\cdot g_k > 0$
for $k\geq 2$ we have a finite set of possible solutions.
For such a $w$ we consider the vectors $t = \beta g_1 + w$ for $\beta \in \NN$.
Since we want to find non-isotropic vectors we have $w\not=0$.
We have $t\cdot t = 2 \beta g_1 \cdot w + w\cdot w$. Since $w\not=0$ we also have  $g_1\cdot w > 0$ and thus a finite number of possibilities to consider.
\end{proof}
\begin{remark}
We postpone the table to Subsection \ref{subsec:enumspaces2} because we will then also add the information about which polarizations will lead to the same modular varieties.
\end{remark}

We note that there are $2$ different polarizations in degree $4$. The first, is given by $h=g_1 + g_2$, the second by $h=g_3$.
Another representation of the first polarization is $h=e+2f$ where $e,f$ are a standard
basis of the hyperbolic plane $U$, i.e. $e^2=f^2=0$ and $e.f=1$. Indeed, if one sets $e = g_2-g_1$ and $f = g_1$
then one gets that $e,f$ define a hyperbolic plane.
This leads to a polarization with base-points (since it has degree $1$ on an elliptic curve).
The second polarization is the one treated by Casnati. Similarly, there are two polarizations in degree $6$, one corresponding to  $h=2g_1+g_2$ or,
alternatively, $h=e+3f$. This is again not base-point free, the other polarization is $h= g_4$ and leads to sextic surfaces in $\PP^3$.
In general, there are the non base-point free polarizations $h = k g_1 + g_2$, or equivalently $h=e+(k+1)f$ of degree $2k+2$.
We shall see later that all polarizations $h = k g_1 + g_2$ lead to equivalent
subgroups $\Gamma_h^+$ and thus to isomorphic modular varieties.
We note that in the (classical) literature non base-point free polarizations are sometimes excluded.
We will return to the connection with the classical cases in more detail in Subsection \ref{subsec:classical}.

\subsection{Enumerating moduli spaces}\label{subsec:enumspaces2}

We will now start enumerating the conjugacy classes of the groups $\Gamma_h^+$. By Section \ref{sec:construction} this is equivalent to enumerating all conjugacy classes of the groups
$\bar \Gamma_h \subset \Orth^+(\FF_2^{10})$.
We shall give detailed information on the groups in Table \ref{table_subgroups1} and Table \ref{table_subgroups2}.

\begin{theorem}\label{teo:numbersubgroupsGammah}\label{teo:87groups}
There are $87$ conjugacy classes of subgroups of the form $\Gamma_h^+$.
\end{theorem}
\begin{proof}
Let  $H = \{ \sum_{i=1}^{10} \alpha_i g_i \mid \alpha_i \in \RR_{\geq 0} \}$ be our chosen fundamental domain of the group $O(U+E_8(-1))$.
The crucial fact which we use is the following:
the stabilizer of a point $x$ in $H$ is generated by the reflections corresponding to the facets of $H$ in which $x$ is contained.
For a proof, see \cite[Theorem 4.8]{Hum}, which in turn refers to \cite[Theorem 1.12.c]{Hum}.
Hence a group $\Gamma_h^+$ is determined by the set of roots to which $h$ is orthogonal to.
There are exactly $10$ roots for the fundamental domain. We note that the isotropic vector $g_1$ cannot represent a polarization. Further, $h$ cannot be orthogonal to all roots (as these span the lattice).
This leaves us with $2^{10} - 1 - 1$ possibilities.

All remaining sets give us potential subgroups $\Gamma_h^+$. We shall now work with the groups $\bar \Gamma_h$, which makes this a finite problem.
These groups can be represented as a permutation group acting on $2^{10}$
elements. By using \cite{gap} we can check when two subgroups are conjugate and thus reduce from
$1022$ to $87$ subgroups. In order to do this practically, one needs to compute suitable invariants.
The level $1$ invariants are the order and the size of the orbits of these groups.
For groups with less than $1000$ elements, we compute all their subgroups and their associated level $1$ invariants.
This gets us a more powerful invariant that would not be possible to compute for the larger groups of this
enumeration.
\end{proof}

In Tables \ref{table_subgroups1} and \ref{table_subgroups2} we provide detailed information on the groups $\bar \Gamma_h$ (and thus equivalently for  $\Gamma_h^+$). For this we use our description that the
group $\Gamma_h^+$ is completely determined by the facets of the fundamental domain containing $h$. This allows us to describe these groups in terms of
admissible subsets of the Dynkin diagram, i.e. subsets which are neither the set of all roots nor consist of only the isotropic vector.
If the number of generating elements is greater than $5$, then we take  the complement of the subset of the Dynkin diagram and indicate this by a line over the set given in the table.
We then give the number of subsets defining the same conjugacy class of subgroups.
The next columns gives the order of $\bar \Gamma_h$ and
we then provide the number of orbits of isotropic vectors and planes in the lattice $N$ with respect to the group $\Gamma_h^+$ (we will return to the latter in more detail in
Subsection \ref{subsec:isotropic}).
Our computations also show that for each group $\Gamma_h$ there is a unique orbit of a vector $h_{\min}$ with $h_{\min}^2$ minimal and $\Gamma_h=\Gamma_{h_{\min}}$ (up to conjugation).
In the last column we provide the $\phi$-invariant of the vector $h_{\min}$ representing the group $\Gamma_h^+$.

Note that this information immediately gives the degree of the maps
\begin{equation*}
\widetilde {\cM}_{\En} \to \cM_{\En,h} \to {\cM}_{\En}.
\end{equation*}
The first is the order of $\bar \Gamma_h$, the latter the index $[\Orth^+(\FF_2^{10}) : \bar \Gamma_h]$, where we recall that
$\left\vert \Orth^+(\FF_2^{10}) \right\vert = 2^{21} \cdot 3^5 \cdot 5^2 \cdot 7\cdot 17 \cdot 31$.

\begin{table}
\begin{center}
{\scriptsize
\begin{tabular}{|c|c|c|c|c|c|c|c|c|}
\hline
Nr & $S$ & $\# S$ &  $|\bar \Gamma_h|$  & $\# I_1$  & $\# I_2$ & $\# I_{12}$ & $\min \deg$ & $\phi(h_{min})$\\
\hline
1 & $\overline{\{0\}}$ & 1 & $2^{14} \cdot 3^{5} \cdot 5^{2} \cdot 7$ & 5 & 9 & 18 & 2 & 1\\
2 & $\overline{\{-1,0\}}$ & 1 & $2^{13} \cdot 3^{5} \cdot 5^{2} \cdot 7$ & 7 & 13 & 28 & 4 & 1\\
3 & $\overline{\{7\}}$ & 1 & $2^{15} \cdot 3^{4} \cdot 5 \cdot 7$ & 5 & 10 & 19 & 4 & 2\\
4 & $\overline{\{1\}}$ & 1 & $2^{11} \cdot 3^{5} \cdot 5 \cdot 7$ & 6 & 14 & 28 & 6 & 2\\
5 & $\overline{\{-1,7\}}$ & 1 & $2^{14} \cdot 3^{2} \cdot 5 \cdot 7$ & 9 & 23 & 49 & 8 & 2\\
6 & $\overline{\{0,1\}}$ & 2 & $2^{11} \cdot 3^{4} \cdot 5 \cdot 7$ & 9 & 23 & 51 & 10 & 2\\
7 & $\overline{\{8\}}$ & 1 & $2^{8} \cdot 3^{4} \cdot 5^{2} \cdot 7$ & 6 & 13 & 28 & 10 & 3\\
8 & $\overline{\{2\}}$ & 1 & $2^{10} \cdot 3^{5} \cdot 5$ & 8 & 22 & 47 & 12 & 3\\
9 & $\overline{\{0,7\}}$ & 1 & $2^{11} \cdot 3^{2} \cdot 5 \cdot 7$ & 11 & 34 & 77 & 14 & 3\\
10 & $\overline{\{7,8\}}$ & 3 & $2^{7} \cdot 3^{4} \cdot 5 \cdot 7$ & 11 & 31 & 74 & 16 & 3\\
11 & $\overline{\{6\}}$ & 1 & $2^{8} \cdot 3^{4} \cdot 5 \cdot 7$ & 8 & 22 & 49 & 18 & 4\\
12 & $\overline{\{1,2\}}$ & 2 & $2^{8} \cdot 3^{5} \cdot 5$ & 12 & 38 & 89 & 18 & 3\\
13 & $\overline{\{-1,0,7\}}$ & 1 & $2^{10} \cdot 3^{2} \cdot 5 \cdot 7$ & 15 & 48 & 116 & 20 & 3\\
14 & $\overline{\{3\}}$ & 1 & $2^{10} \cdot 3^{2} \cdot 5^{2}$ & 10 & 34 & 76 & 20 & 4\\
15 & $\overline{\{-1,0,1\}}$ & 1 & $2^{10} \cdot 3^{4} \cdot 5 \cdot 7$ & 12 & 32 & 76 & 22 & 3\\
16 & $\overline{\{1,7\}}$ & 1 & $2^{10} \cdot 3^{3} \cdot 5$ & 13 & 49 & 113 & 22 & 4\\
17 & $\overline{\{0,8\}}$ & 3 & $2^{8} \cdot 3^{2} \cdot 5 \cdot 7$ & 14 & 51 & 125 & 24 & 4\\
18 & $\overline{\{0,2\}}$ & 1 & $2^{9} \cdot 3^{4} \cdot 5$ & 13 & 46 & 108 & 26 & 4\\
19 & $\overline{\{2,7\}}$ & 3 & $2^{10} \cdot 3^{2} \cdot 5$ & 16 & 67 & 162 & 28 & 4\\
20 & $\overline{\{4\}}$ & 1 & $2^{7} \cdot 3^{3} \cdot 5^{2}$ & 11 & 41 & 96 & 30 & 5\\
21 & $\overline{\{0,1,7\}}$ & 2 & $2^{10} \cdot 3^{2} \cdot 5$ & 19 & 83 & 204 & 30 & 4\\
22 & $\overline{\{-1,0,8\}}$ & 4 & $2^{7} \cdot 3^{2} \cdot 5 \cdot 7$ & 19 & 74 & 191 & 32 & 4\\
23 & $\overline{\{0,1,2\}}$ & 3 & $2^{8} \cdot 3^{4} \cdot 5$ & 17 & 64 & 159 & 34 & 4\\
24 & $\overline{\{5,8\}}$ & 2 & $2^{5} \cdot 3^{3} \cdot 5 \cdot 7$ & 16 & 68 & 171 & 34 & 5\\
25 & $\overline{\{5,7\}}$ & 3 & $2^{6} \cdot 3^{2} \cdot 5 \cdot 7$ & 18 & 85 & 213 & 36 & 5\\
26 & $\overline{\{0,3\}}$ & 2 & $2^{9} \cdot 3^{2} \cdot 5$ & 18 & 87 & 213 & 38 & 5\\
27 & $\overline{\{3,8\}}$ & 3 & $2^{6} \cdot 3^{2} \cdot 5^{2}$ & 19 & 91 & 233 & 40 & 5\\
28 & $\overline{\{5\}}$ & 1 & $2^{6} \cdot 3^{3} \cdot 5 \cdot 7$ & 12 & 49 & 115 & 42 & 6\\
29 & $\overline{\{1,2,7\}}$ & 5 & $2^{8} \cdot 3^{2} \cdot 5$ & 24 & 124 & 320 & 42 & 5\\
30 & $\overline{\{3,7\}}$ & 1 & $2^{9} \cdot 3^{2} \cdot 5$ & 18 & 85 & 208 & 44 & 6\\
31 & $\overline{\{0,1,8\}}$ & 9 & $2^{5} \cdot 3^{2} \cdot 5 \cdot 7$ & 24 & 122 & 322 & 44 & 5\\
32 & $\overline{\{2,8\}}$ & 3 & $2^{7} \cdot 3^{3} \cdot 5$ & 18 & 85 & 215 & 46 & 6\\
33 & $\overline{\{-1,5\}}$ & 4 & $2^{6} \cdot 3^{3} \cdot 5$ & 21 & 115 & 294 & 48 & 6\\
34 & $\overline{\{-1,0,1,7\}}$ & 1 & $2^{9} \cdot 3^{2} \cdot 5$ & 25 & 118 & 305 & 50 & 5\\
35 & $\overline{\{2,6\}}$ & 4 & $2^{7} \cdot 3^{2} \cdot 5$ & 23 & 136 & 351 & 52 & 6\\
36 & $\overline{\{0,2,7\}}$ & 4 & $2^{9} \cdot 3 \cdot 5$ & 26 & 151 & 388 & 54 & 6\\
37 & $\overline{\{2,3,7\}}$ & 2 & $2^{9} \cdot 3^{2}$ & 29 & 179 & 466 & 56 & 6\\
38 & $\overline{\{1,2,8\}}$ & 8 & $2^{5} \cdot 3^{3} \cdot 5$ & 28 & 167 & 448 & 58 & 6\\
39 & $\overline{\{0,2,8\}}$ & 9 & $2^{6} \cdot 3^{2} \cdot 5$ & 31 & 207 & 553 & 60 & 6\\
40 & $\overline{\{-1,0,1,8\}}$ & 5 & $2^{4} \cdot 3^{2} \cdot 5 \cdot 7$ & 32 & 181 & 495 & 62 & 6\\
41 & $\overline{\{2,7,8\}}$ & 13 & $2^{6} \cdot 3^{2} \cdot 5$ & 31 & 200 & 540 & 64 & 6\\
\hline
\end{tabular}
}
\caption{The group $\bar \Gamma_h$. For each group we give one representative as a subset $S$ of the diagram. If the number of generating elements is greater than $5$, then we take the complement
and denote this by $\overline S$.
  We then give the number $\# S$ of subsets leading to the same group. 
  The next column is  the order of $\bar \Gamma_h$.
  The next three columns give the numbers $\#  I_1$, $\#  I_2$ and $\# I_{12}$ of orbits of isotropic lines, planes and flags in $N$.
  The penultimate column gives the degree of the (unique) smallest realization $h_{\min}$ having this group and the last column gives $\phi(h_{\min})$ (Part 1).}\label{table_subgroups1}
\end{center}
\end{table}

\begin{table}
\begin{center}
{\scriptsize
\begin{tabular}{|c|c|c|c|c|c|c|c|c|}
\hline
Nr & $S$ & $\# S$ &  $|\bar \Gamma_h|$  & $\# I_1$  & $\# I_2$ & $\# I_{12}$ & $\min \deg$ & $\phi(h_{min})$\\
\hline
42 & $\overline{\{1,4\}}$ & 1 & $2^{5} \cdot 3^{3} \cdot 5$ & 24 & 149 & 389 & 66 & 7\\
43 & $\overline{\{0,1,2,7\}}$ & 7 & $2^{8} \cdot 3 \cdot 5$ & 34 & 219 & 585 & 66 & 6\\
44 & $\overline{\{0,5\}}$ & 2 & $2^{6} \cdot 3^{2} \cdot 5$ & 27 & 187 & 487 & 68 & 7\\
45 & $\overline{\{-1,0,1,2\}}$ & 1 & $2^{7} \cdot 3^{4} \cdot 5$ & 22 & 92 & 238 & 70 & 6\\
46 & $\overline{\{0,3,7\}}$ & 2 & $2^{8} \cdot 3^{2}$ & 33 & 239 & 629 & 70 & 7\\
47 & $\overline{\{0,5,8\}}$ & 19 & $2^{5} \cdot 3^{2} \cdot 5$ & 36 & 274 & 746 & 76 & 7\\
48 & $\overline{\{0,1,2,8\}}$ & 17 & $2^{5} \cdot 3^{2} \cdot 5$ & 41 & 307 & 849 & 78 & 7\\
49 & $\overline{\{1,5\}}$ & 2 & $2^{6} \cdot 3^{3}$ & 30 & 231 & 608 & 84 & 8\\
50 & $\overline{\{1,2,3,7\}}$ & 3 & $2^{7} \cdot 3^{2}$ & 44 & 354 & 967 & 84 & 7\\
51 & $\overline{\{2,6,8\}}$ & 7 & $2^{7} \cdot 3^{2}$ & 38 & 306 & 832 & 88 & 8\\
52 & $\overline{\{0,5,7\}}$ & 6 & $2^{6} \cdot 3 \cdot 5$ & 40 & 342 & 927 & 92 & 8\\
53 & $\overline{\{1,2,7,8\}}$ & 20 & $2^{4} \cdot 3^{2} \cdot 5$ & 48 & 413 & 1159 & 96 & 8\\
54 & $\overline{\{0,1,5\}}$ & 14 & $2^{6} \cdot 3^{2}$ & 45 & 429 & 1175 & 100 & 8\\
55 & $\overline{\{0,1,3,7\}}$ & 3 & $2^{8} \cdot 3$ & 48 & 435 & 1184 & 102 & 8\\
56 & $\overline{\{2,3,7,8\}}$ & 11 & $2^{6} \cdot 3^{2}$ & 51 & 463 & 1298 & 104 & 8\\
57 & $\overline{\{1,5,8\}}$ & 4 & $2^{5} \cdot 3^{3}$ & 40 & 341 & 937 & 106 & 9\\
58 & $\overline{\{0,5,7,8\}}$ & 30 & $2^{5} \cdot 3 \cdot 5$ & 53 & 512 & 1433 & 108 & 8\\
59 & $\overline{\{-1,0,1,2,8\}}$ & 6 & $2^{4} \cdot 3^{2} \cdot 5$ & 54 & 465 & 1315 & 110 & 8\\
60 & $\overline{\{-1,0,1,2,7\}}$ & 2 & $2^{7} \cdot 3 \cdot 5$ & 44 & 325 & 890 & 114 & 8\\
61 & $\overline{\{1,2,5\}}$ & 4 & $2^{4} \cdot 3^{3}$ & 48 & 484 & 1339 & 120 & 9\\
62 & $\overline{\{0,1,5,8\}}$ & 46 & $2^{5} \cdot 3^{2}$ & 60 & 649 & 1832 & 124 & 9\\
63 & $\overline{\{0,1,2,7,8\}}$ & 32 & $2^{4} \cdot 3 \cdot 5$ & 70 & 782 & 2235 & 132 & 9\\
64 & $\overline{\{0,1,2,3,7\}}$ & 4 & $2^{7} \cdot 3$ & 63 & 653 & 1825 & 138 & 9\\
65 & $\overline{\{0,2,5\}}$ & 3 & $2^{5} \cdot 3^{2}$ & 54 & 611 & 1685 & 140 & 10\\
66 & $\overline{\{0,1,5,7\}}$ & 20 & $2^{6} \cdot 3$ & 67 & 814 & 2291 & 148 & 10\\
67 & $\overline{\{1,2,3,7,8\}}$ & 30 & $2^{4} \cdot 3^{2}$ & 80 & 1002 & 2884 & 156 & 10\\
68 & $\overline{\{0,1,2,5\}}$ & 30 & $2^{4} \cdot 3^{2}$ & 72 & 929 & 2637 & 160 & 10\\
69 & $\overline{\{0,1,5,7,8\}}$ & 57 & $2^{5} \cdot 3$ & 89 & 1252 & 3599 & 180 & 11\\
70 & $\overline{\{1,2,5,8\}}$ & 5 & $2^{3} \cdot 3^{3}$ & 64 & 737 & 2097 & 184 & 12\\
71 & $\overline{\{0,2,5,7\}}$ & 11 & $2^{5} \cdot 3$ & 81 & 1174 & 3323 & 196 & 12\\
72 & $\{4,6,3,5\}$ & 8 & $2^{3} \cdot 3 \cdot 5$ & 92 & 1210 & 3503 & 198 & 11\\
73 & $\overline{\{0,1,2,5,8\}}$ & 48 & $2^{3} \cdot 3^{2}$ & 96 & 1440 & 4166 & 208 & 12\\
74 & $\overline{\{0,1,2,5,7\}}$ & 64 & $2^{4} \cdot 3$ & 108 & 1818 & 5251 & 220 & 12\\
75 & $\{4,6,-1,5\}$ & 44 & $2^{4} \cdot 3$ & 118 & 1953 & 5690 & 228 & 12\\
76 & $\{4,6,8,5\}$ & 1 & $2^{6} \cdot 3$ & 82 & 996 & 2829 & 234 & 12\\
77 & $\{4,6,3,7\}$ & 19 & $2^{2} \cdot 3^{2}$ & 128 & 2263 & 6625 & 260 & 13\\
78 & $\overline{\{0,1,3,5,7\}}$ & 9 & $2^{5}$ & 122 & 2306 & 6647 & 280 & 14\\
79 & $\{4,6,-1,3\}$ & 99 & $2^{3} \cdot 3$ & 144 & 2856 & 8357 & 292 & 14\\
80 & $\{4,6,-1,2\}$ & 39 & $2^{4}$ & 163 & 3626 & 10599 & 340 & 15\\
81 & $\{4,6,5\}$ & 9 & $2^{3} \cdot 3$ & 156 & 3074 & 9031 & 342 & 15\\
82 & $\{4,6,3\}$ & 54 & $2^{2} \cdot 3$ & 192 & 4532 & 13369 & 380 & 16\\
83 & $\{4,6,-1\}$ & 57 & $2^{3}$ & 218 & 5766 & 17003 & 460 & 18\\
84 & $\{4,3\}$ & 9 & $2 \cdot 3$ & 256 & 7242 & 21471 & 532 & 19\\
85 & $\{4,6\}$ & 36 & $2^{2}$ & 292 & 9246 & 27411 & 580 & 20\\
86 & $\{4\}$ & 10 & $2$ & 392 & 14926 & 44387 & 820 & 24\\
87 & $\emptyset$ & 1 & $1$ & 528 & 24242 & 72199 & 1240 & 30\\
\hline
\end{tabular}
}
\caption{Same as Table \ref{table_subgroups1} (Part 2)}\label{table_subgroups2}
\end{center}
\end{table}

\newpage

\subsection{Degree of the polarization and number of moduli spaces}\label{subsec:comparisondegree}

As the degree of the polarization increases, the number of inequivalent polarizations will also grow. At the same time, the number of conjugacy classes of groups $\bar \Gamma_h$, and hence of
modular varieties $\cM_{\En,h}$ is limited by $87$. This means that inequivalent polarizations must give rise to isomorphic modular varieties. We will now discuss this in more detail.

In Table \ref{ListVectorsNormAtMost30} we list a representative for each polarization class in given low degree. We also enumerate the different conjugacy classes of the groups $\bar \Gamma_h$.
The entries $2 : \{g_{1}+g_{2}\}$ and $2 : \{2g_{1}+g_{2}\}$, for example, mean that the degree $4$ polarization $g_{1}+g_{2}$ and the degree $6$ polarization $2g_{1}+g_{2}$ define conjugate
subgroups $\bar \Gamma_h$.

Table \ref{ListNumberPolarizationsModuli} gives the following information for each degree $2d=2, \ldots, 72$ and
corresponding genus $g=2, \ldots, 37$: the first line shows the number $\# h$ of orbits of primitive vectors in a given degree.
We note that these numbers agree exactly with the corresponding list in \cite[Appendix]{CDGK}.
The second line $\# \bar \Gamma_h$ gives the number of conjugacy classes of groups $\bar \Gamma_h$ for given degree. We note that  $\# \bar \Gamma_h \leq \# h$ and that strict inequality will occur when different
orbits of primitive vectors $h$ give rise to conjugate subgroups $\bar \Gamma_h$.
This phenomenon first appears in degree $12$ where the stabilizer groups of the two polarizations $5g_1 + g_2$ and $g_1 + 2g_2$ actually agree.
As the degree grows, the number of orbits $\# h$ of will grow much faster than $\# \bar \Gamma_h$.
We note that the numbers in this list agree with those given in  \cite[Corollary 5.6]{GH3}.

In the next two lines we compare how in degree at most $2d$ the number of orbits and the number of conjugacy classes increase.
We see that we have found $312$ different classes of polarizations and $46$ different conjugacy classes of  groups $\bar \Gamma_h$ in degree $\leq 72$.
The number of subgroups will finally stabilize to $87$, by
Theorem \ref{teo:numbersubgroupsGammah}. This happens in degree $1240$,
 which is a lower limit by Theorem \ref{teo:degreemaximalgroup}.
 The above lists can easily be extended to higher degree (genus) using the programs we have.
The norm $2d = 72$ is the first one for which there is no new group occurring.

\begin{table}
\begin{center}

\begin{tabular}{|c|c|}
\hline
$\deg$ & polarizations \\
\hline
2 & 1 : $\{g_{2}\}$.\\
\hline
4 & 2 : $\{g_{1}+g_{2}\}$, 3 : $\{g_{3}\}$.\\
\hline
6 & 2 : $\{2g_{1}+g_{2}\}$, 4 : $\{g_{4}\}$.\\
\hline
8 & 2 : $\{3g_{1}+g_{2}\}$, 5 : $\{g_{1}+g_{3}\}$.\\
\hline
10 & 2 : $\{4g_{1}+g_{2}\}$, 6 : $\{g_{1}+g_{4}\}$, 7 : $\{g_{5}\}$.\\
\hline
12 & 2 : $\{5g_{1}+g_{2}$, $g_{1}+2g_{2}\}$, 5 : $\{2g_{1}+g_{3}\}$, 8 : $\{g_{6}\}$.\\
\hline
14 & 2 : $\{6g_{1}+g_{2}\}$, 6 : $\{2g_{1}+g_{4}\}$, 9 : $\{g_{2}+g_{3}\}$.\\
\hline
16 & 2 : $\{7g_{1}+g_{2}\}$, 5 : $\{3g_{1}+g_{3}\}$, 6 : $\{g_{2}+g_{4}\}$, 10 : $\{g_{1}+g_{5}\}$.\\
\hline
18 & 2 : $\{8g_{1}+g_{2}\}$, 6 : $\{3g_{1}+g_{4}\}$, 11 : $\{g_{7}\}$, 12 : $\{g_{1}+g_{6}\}$.\\
\hline
20 & 2 : $\{9g_{1}+g_{2}$, $3g_{1}+2g_{2}\}$, 5 : $\{4g_{1}+g_{3}\}$, 13 : $\{g_{1}+g_{2}+g_{3}\}$, \\
    & 14 : $\{g_{8}\}$.\\
\hline
22 & 2 : $\{10g_{1}+g_{2}\}$, 6 : $\{4g_{1}+g_{4}\}$, 10 : $\{2g_{1}+g_{5}\}$, 15 : $\{g_{1}+g_{2}+g_{4}\}$, \\
    & 16 : $\{g_{3}+g_{4}\}$.\\
\hline
24 & 2 : $\{11g_{1}+g_{2}$, $g_{1}+3g_{2}\}$, 5 : $\{5g_{1}+g_{3}$, $g_{1}+2g_{3}\}$, \\
    & 12 : $\{2g_{1}+g_{6}\}$, 17 : $\{g_{2}+g_{5}\}$.\\
\hline
26 & 2 : $\{12g_{1}+g_{2}\}$, 6 : $\{5g_{1}+g_{4}\}$, 13 : $\{2g_{1}+g_{2}+g_{3}\}$, \\
    & 17 : $\{g_{1}+g_{7}\}$, 18 : $\{g_{2}+g_{6}\}$.\\
\hline
28 & 2 : $\{13g_{1}+g_{2}$, $5g_{1}+2g_{2}\}$, 5 : $\{6g_{1}+g_{3}\}$, 9 : $\{2g_{2}+g_{3}\}$, \\
    & 10 : $\{3g_{1}+g_{5}$, $g_{3}+g_{5}\}$, 15 : $\{2g_{1}+g_{2}+g_{4}\}$, 19 : $\{g_{1}+g_{8}\}$.\\
\hline
30 & 2 : $\{14g_{1}+g_{2}$, $2g_{1}+3g_{2}\}$, 6 : $\{6g_{1}+g_{4}$, $2g_{2}+g_{4}\}$, \\
    & 12 : $\{3g_{1}+g_{6}\}$, 20 : $\{g_{9}\}$, 21 : $\{g_{1}+g_{3}+g_{4}\}$.\\
\hline
all & 1 : $\{g_{2}\}$, 2 : $\{g_{1}+g_{2}$, $2g_{1}+g_{2}$, $3g_{1}+g_{2}$, $4g_{1}+g_{2}$, \\
    & $5g_{1}+g_{2}$, $6g_{1}+g_{2}$, $7g_{1}+g_{2}$, $8g_{1}+g_{2}$, $9g_{1}+g_{2}$, \\
    & $10g_{1}+g_{2}$, $11g_{1}+g_{2}$, $12g_{1}+g_{2}$, $13g_{1}+g_{2}$, $14g_{1}+g_{2}$, \\
    & $g_{1}+2g_{2}$, $3g_{1}+2g_{2}$, $5g_{1}+2g_{2}$, $g_{1}+3g_{2}$, $2g_{1}+3g_{2}\}$, \\
    & 3 : $\{g_{3}\}$, 4 : $\{g_{4}\}$, 5 : $\{g_{1}+g_{3}$, $2g_{1}+g_{3}$, $3g_{1}+g_{3}$, \\
    & $4g_{1}+g_{3}$, $5g_{1}+g_{3}$, $6g_{1}+g_{3}$, $g_{1}+2g_{3}\}$, 6 : $\{g_{1}+g_{4}$, \\
    & $2g_{1}+g_{4}$, $3g_{1}+g_{4}$, $4g_{1}+g_{4}$, $5g_{1}+g_{4}$, $6g_{1}+g_{4}$, \\
    & $g_{2}+g_{4}$, $2g_{2}+g_{4}\}$, 7 : $\{g_{5}\}$, 8 : $\{g_{6}\}$, 9 : $\{g_{2}+g_{3}$, \\
    & $2g_{2}+g_{3}\}$, 10 : $\{g_{1}+g_{5}$, $2g_{1}+g_{5}$, $3g_{1}+g_{5}$, $g_{3}+g_{5}\}$, \\
    & 11 : $\{g_{7}\}$, 12 : $\{g_{1}+g_{6}$, $2g_{1}+g_{6}$, $3g_{1}+g_{6}\}$, 13 : $\{g_{1}+g_{2}+g_{3}$, \\
    & $2g_{1}+g_{2}+g_{3}\}$, 14 : $\{g_{8}\}$, 15 : $\{g_{1}+g_{2}+g_{4}$, $2g_{1}+g_{2}+g_{4}\}$, \\
    & 16 : $\{g_{3}+g_{4}\}$, 17 : $\{g_{2}+g_{5}$, $g_{1}+g_{7}\}$, 18 : $\{g_{2}+g_{6}\}$, \\
    & 19 : $\{g_{1}+g_{8}\}$, 20 : $\{g_{9}\}$, 21 : $\{g_{1}+g_{3}+g_{4}\}$.\\
\hline
\end{tabular}

\newpage

\caption{The primitive vectors of norms at most $30$ in the fundamental domain expressed in the basis $(g_i)$. The vectors are first grouped by norm and then all together. If two or more vectors give rise
to conjugate groups, then they are grouped in a set whose index corresponds to the one of Tables \ref{table_subgroups1} and  \ref{table_subgroups2}.}\label{ListVectorsNormAtMost30}

\end{center}
\end{table}

\begin{table}
\begin{center}
\begin{tabular}{|c|cccccccccccc|}
\hline
$g$  & 2 & 3 & 4 & 5 & 6 & 7 & 8 & 9 & 10 & 11 & 12 & 13\\
$2d$  & 2 & 4 & 6 & 8 & 10 & 12 & 14 & 16 & 18 & 20 & 22 & 24\\
\hline
$\# h$  & 1 & 2 & 2 & 2 & 3 & 4 & 3 & 4 & 4 & 5 & 5 & 6\\
$\# \bar \Gamma_h$  & 1 & 2 & 2 & 2 & 3 & 3 & 3 & 4 & 4 & 4 & 5 & 4\\
$\# h, h^2 \leq 2d$  & 1 & 3 & 5 & 7 & 10 & 14 & 17 & 21 & 25 & 30 & 35 & 41\\
$\# \bar \Gamma_{h}, h^2\leq 2d$  & 1 & 3 & 4 & 5 & 7 & 8 & 9 & 10 & 12 & 14 & 16 & 17\\
\hline
\hline
$g$  & 14 & 15 & 16 & 17 & 18 & 19 & 20 & 21 & 22 & 23 & 24 & 25\\
$2d$  & 26 & 28 & 30 & 32 & 34 & 36 & 38 & 40 & 42 & 44 & 46 & 48\\
\hline
$\# h$  & 5 & 8 & 7 & 6 & 8 & 8 & 7 & 10 & 10 & 10 & 11 & 11\\
$\# \bar \Gamma_h$  & 5 & 6 & 5 & 6 & 8 & 6 & 7 & 7 & 7 & 7 & 10 & 7\\
$\# h, h^2 \leq 2d$  & 46 & 54 & 61 & 67 & 75 & 83 & 90 & 100 & 110 & 120 & 131 & 142\\
$\# \bar \Gamma_{h}, h^2\leq 2d$  & 18 & 19 & 21 & 22 & 24 & 25 & 26 & 27 & 29 & 31 & 32 & 33\\
\hline
\hline
$g$  & 26 & 27 & 28 & 29 & 30 & 31 & 32 & 33 & 34 & 35 & 36 & 37\\
$2d$  & 50 & 52 & 54 & 56 & 58 & 60 & 62 & 64 & 66 & 68 & 70 & 72\\
\hline
$\# h$  & 9 & 14 & 11 & 12 & 14 & 16 & 13 & 15 & 16 & 16 & 18 & 16\\
$\# \bar \Gamma_h$  & 9 & 10 & 11 & 8 & 12 & 9 & 11 & 14 & 11 & 11 & 12 & 10\\
$\# h, h^2 \leq 2d$  & 151 & 165 & 176 & 188 & 202 & 218 & 231 & 246 & 262 & 278 & 296 & 312\\
$\# \bar \Gamma_{h}, h^2\leq 2d$  & 34 & 35 & 36 & 37 & 38 & 39 & 40 & 41 & 43 & 44 & 46 & 46\\
\hline
\end{tabular}
\caption{The number of primitive polarizations and modular groups}\label{ListNumberPolarizationsModuli}
\end{center}
\end{table}

\newpage

\subsection{Connection with the $\phi$-invariant} \label{subsec:phi}
In  \cite[Appendix]{CDGK} Ciliberto et al. gave a systematic enumeration of polarizations for genus up to $30$. We will now match their enumeration with our results.

A crucial role in \cite{CDGK} is played by the minimal degree of a polarization $h$
on an effective elliptic curve $E$, namely
\begin{equation*}
\phi (h) = \min \{h.E \mid E^2=0, E >0 \}.
\end{equation*}
Using this parameter and the genus they consider the moduli spaces  $\widehat{\cE}_{g,\phi}$ of polarized Enriques surfaces with given $\phi$ and genus $g$.
The crucial tool in their enumeration is the notion of {\em decomposition type} given in \cite[Definition 4.13]{CDGK}. This can lead to more than one component of
a moduli space $\widehat{\cE}_{g,\phi}$, in our cases denoted by $\widehat{\cE}^{(I)}_{g,\phi}$ and $\widehat{\cE}^{(II)}_{g,\phi}$.

Another essential technical tool is the notion of an {\em isotropic $10$-sequence} as defined in \cite[Definition 3.2]{CDGK}, and which goes back
to  Cossec and Dolgachev \cite[p. 122]{CD}.
This is a collection of effective isotropic classes $E_i, i=1, \ldots ,10$ which span the lattice $M(1/2)=U+E_8(-1)$ over the rationals with the
additional property that $E_i.E_j=1$ for $i \neq j$.
By  \cite[Lemma 3.4]{CDGK}, see also \cite[Lemma 1.6.2 (i)]{Cos} or \cite[Corollary 2.5.5]{CD}, an isotropic $10$-sequence has the further property that $\sum E_i$ is $3$-divisible, i.e. there is a divisor $D$
with $3D \equiv \sum_i E_i$.
By the defining property of an isotropic $10$-sequence it then follows that $D^2=10$. This observation also implies that the $E_i$
form a $\QQ$-basis, but not a $\ZZ$-basis of $M(1/2)$.

We further note that the list of \cite[Appendix]{CDGK} also contains non-primitive numerical polarizations, which we disregard in our approach since they do not lead to new moduli spaces.

We now want to provide a precise matching between the (components of the) moduli spaces $\widehat{\cE}_{g,\phi}$ and our modular varieties $\cM_{\En,h}$.
To do this, we first introduce a new integral
basis $u_i$,  $i= 1, \ldots ,10$ of the lattice $M(1/2)=U + E_8(-1)$ with Gram matrix
\begin{equation}\label{basisui}
G_u=
\left(\begin{array}{cccccccccc}
0 &   1 &   1 &   1 &   3 &   1 &   1 &   1 &   1 &   1\\
 1 &   0 &   1 &   1 &   3 &   1 &   1 &   1 &   1 &   1 \\
 1 &   1 &   0 &   1 &   3 &   1 &   1 &   1 &   1 &   1 \\
  1 &   1 &   1 &   0 &   3 &   1 &   1 &   1 &   1 &   1\\
  3 &   3 &   3 &   3 &  10 &   3 &   3 &   3 &   3 &   3\\
  1 &  1 &   1 &   1 &   3 &   0 &   1 &   1 &   1 &   1\\
  1 &   1 &   1 &   1 &   3 &   1 &   0 &   1 &   1 &   1 \\
  1 &   1 &   1 &   1 &   3 &   1 &   1&   0 &   1 &   1 \\
  1 &   1 &   1 &   1&   3 &   1 &   1 &   1 &   0 &   1 \\
  1 &   1 &   1 &   1 &   3 &   1 &  1 &   1 &   1 &   0
\end{array}\right)
\end{equation}

The equivalence between the basis $u_i$ and the standard basis of the lattice $U + E_8(-1)$ is provided by the matrix
\begin{equation}\label{test}
T=
\left(\begin{array}{cccccccccc}
  0 &  -1 &   1 &   0 &   0 &   0 &   1&   0 &  0 &   1 \\
  1 & -1 &   1 &   0 &   0 &   0 &   1 &   0 &   0 &   0 \\
  0 &  -1 &   1 & -1 &   1 &  -1 &   0 &   0 &   0 &   0 \\
  1 &  -1 &   2 &   1&  -2 &   1 &   2 &   1 &   1 &   1 \\
  -1 &  -1 &  -1 &  -1 &   3 &  -1 &  -1 &  -1 &  -2 &  -1 \\
   0 &   0 &   0 &   0 &   0 &   0 &   0 &   1 &  -1 &   0 \\
   0 &   1 &  -1 &  -1 &   0 &   0 &   0 &   0 &   0 &   0 \\
   0 &  -1 &   1 &   0 &   0 &   1 &   0 &   0 &   0 &   0 \\
   0 &   0 &   0 &  -1 &   0 &  0 &   0 &   1&   0 &   0 \\
   0 &   0 &   1 &   1 &  -1 &   0 &   1 &   0 &   1 &   0
\end{array}\right)
\end{equation}
meaning that $T^tG_u T=G$ with $G$ as in \eqref{equ:G}.
The columns of the matrix $T$ are the vectors $u_i$.

The isotropic $10$-sequence which we are looking for can be obtained from the basis $u_i$ by means of the transition matrix
\begin{equation}\label{isotropic10}
I =
\left(\begin{array}{cccccccccc}
 -1 &  -1 &  -1 &  -1 &   3 &  -1 &  -&  -1 &  -1 &  -1\\
  1 &   0 &   0 &   0 &   0 &   0 &   0 &   0 &   0 &  0 \\
  0 &   1 &   0 &   0 &   0 &   0 &   0 &   0, &   0 &   0 \\
  0 &   0 &   1 &   0 &   0 &   0 &   0 &   0 &   0 &   0 \\
  0&   0&   0&   1&   0&   0&   0&   0&   0&   0 \\
  0&   0&   0&   0&   0&   1&   0&   0&   0&   0 \\
  0&   0&   0&   0&   0&   0&   1&   0&   0&   0 \\
  0&   0&   0&   0&   0&   0&   0&   1&   0&   0 \\
  0&   0&   0&   0 &   0&   0&   0&   0&   1&   0 \\
  0&   0&   0&   0&   0,&   0&   0&   0&   0&   1 \\
\end{array}\right)
\end{equation}
One sees immediately that $\det(I)=3$. The Gram matrix of the vectors $E_i$ is given by
\begin{equation}\label{GramE}
G_E = (1)_{1\leq i,j\leq 10} - I_n
\end{equation}
which immediately shows that the defining conditions of an isotropic $10$-sequence are satisfied.

We can now make the correspondence between the varieties $\cM_{\En,h}$ and $\widehat{\cE}_{g,\phi}$ for $g \leq 30$ explicit.
The result is given in Tables \ref{Expression_Isotropic_Vectors_Part1} and \ref{Expression_Isotropic_Vectors_Part2}. Note that we also list non-primitive polarizations here in order to have a full matching with
 \cite[Appendix]{CDGK}.

\begin{table}
\begin{center}
{\scriptsize
\begin{tabular}{|cccc|cccc|cccc|}
\hline
$g$  &  $\phi$  &   CDGK   &   DH     &    $g$  &  $\phi$  &  CDGK  &  DH   &    $g$  &   $\phi$  &  CDGK  &  DH\\
\hline
2  &  1  &  $\widehat{E}_{2,1}$  &  $g_{2}$    &    3  &  1  &  $\widehat{E}_{3,1}$  &  $g_{1}+g_{2}$    &    3  &  2  &  $\widehat{E}_{3,2}$  &  $g_{3}$\\
4  &  1  &  $\widehat{E}_{4,1}$  &  $2g_{1}+g_{2}$    &    4  &  2  &  $\widehat{E}_{4,2}$  &  $g_{4}$    &    5  &  1  &  $\widehat{E}_{5,1}$  &  $3g_{1}+g_{2}$\\
5  &  2  &  $\widehat{E}_{5,2}^{(I)}$  &  $g_{1}+g_{3}$    &    5  &  2  &  $\widehat{E}_{5,2}^{(II)}$  &  $2g_{2}$    &    6  &  1  &  $\widehat{E}_{6,1}$  &  $4g_{1}+g_{2}$\\
6  &  2  &  $\widehat{E}_{6,2}$  &  $g_{1}+g_{4}$    &    6  &  3  &  $\widehat{E}_{6,3}$  &  $g_{5}$    &    7  &  1  &  $\widehat{E}_{7,1}$  &  $5g_{1}+g_{2}$\\
7  &  2  &  $\widehat{E}_{7,2}^{(I)}$  &  $2g_{1}+g_{3}$    &    7  &  2  &  $\widehat{E}_{7,2}^{(II)}$  &  $g_{1}+2g_{2}$    &    7  &  3  &  $\widehat{E}_{7,3}$  &  $g_{6}$\\
8  &  1  &  $\widehat{E}_{8,1}$  &  $6g_{1}+g_{2}$    &    8  &  2  &  $\widehat{E}_{8,2}$  &  $2g_{1}+g_{4}$    &    8  &  3  &  $\widehat{E}_{8,3}$  &  $g_{2}+g_{3}$\\
9  &  1  &  $\widehat{E}_{9,1}$  &  $7g_{1}+g_{2}$    &    9  &  2  &  $\widehat{E}_{9,2}^{(I)}$  &  $3g_{1}+g_{3}$    &    9  &  2  &  $\widehat{E}_{9,2}^{(II)}$  &  $2g_{1}+2g_{2}$\\
9  &  3  &  $\widehat{E}_{9,3}^{(I)}$  &  $g_{1}+g_{5}$    &    9  &  3  &  $\widehat{E}_{9,3}^{(II)}$  &  $g_{2}+g_{4}$    &    9  &  4  &  $\widehat{E}_{9,4}$  &  $2g_{3}$\\
10  &  1  &  $\widehat{E}_{10,1}$  &  $8g_{1}+g_{2}$    &    10  &  2  &  $\widehat{E}_{10,2}$  &  $3g_{1}+g_{4}$    &    10  &  3  &  $\widehat{E}_{10,3}^{(I)}$  &  $g_{1}+g_{6}$\\
10  &  3  &  $\widehat{E}_{10,3}^{(II)}$  &  $3g_{2}$    &    10  &  4  &  $\widehat{E}_{10,4}$  &  $g_{7}$    &    11  &  1  &  $\widehat{E}_{11,1}$  &  $9g_{1}+g_{2}$\\
11  &  2  &  $\widehat{E}_{11,2}^{(I)}$  &  $4g_{1}+g_{3}$    &    11  &  2  &  $\widehat{E}_{11,2}^{(II)}$  &  $3g_{1}+2g_{2}$    &    11  &  3  &  $\widehat{E}_{11,3}$  &  $g_{1}+g_{2}+g_{3}$\\
11  &  4  &  $\widehat{E}_{11,4}$  &  $g_{8}$    &    12  &  1  &  $\widehat{E}_{12,1}$  &  $10g_{1}+g_{2}$    &    12  &  2  &  $\widehat{E}_{12,2}$  &  $4g_{1}+g_{4}$\\
12  &  3  &  $\widehat{E}_{12,3}^{(I)}$  &  $g_{1}+g_{2}+g_{4}$    &    12  &  3  &  $\widehat{E}_{12,3}^{(II)}$  &  $2g_{1}+g_{5}$    &    12  &  4  &  $\widehat{E}_{12,4}$  &  $g_{3}+g_{4}$\\
13  &  1  &  $\widehat{E}_{13,1}$  &  $11g_{1}+g_{2}$    &    13  &  2  &  $\widehat{E}_{13,2}^{(I)}$  &  $5g_{1}+g_{3}$    &    13  &  2  &  $\widehat{E}_{13,2}^{(II)}$  &  $4g_{1}+2g_{2}$\\
13  &  3  &  $\widehat{E}_{13,3}^{(I)}$  &  $2g_{1}+g_{6}$    &    13  &  3  &  $\widehat{E}_{13,3}^{(II)}$  &  $g_{1}+3g_{2}$    &    13  &  4  &  $\widehat{E}_{13,4}^{(I)}$  &  $g_{2}+g_{5}$\\
13  &  4  &  $\widehat{E}_{13,4}^{(II)}$  &  $2g_{4}$    &    13  &  4  &  $\widehat{E}_{13,4}^{(III)}$  &  $g_{1}+2g_{3}$    &    14  &  1  &  $\widehat{E}_{14,1}$  &  $12g_{1}+g_{2}$\\
14  &  2  &  $\widehat{E}_{14,2}$  &  $5g_{1}+g_{4}$    &    14  &  3  &  $\widehat{E}_{14,3}$  &  $2g_{1}+g_{2}+g_{3}$    &    14  &  4  &  $\widehat{E}_{14,4}^{(I)}$  &  $g_{2}+g_{6}$\\
14  &  4  &  $\widehat{E}_{14,4}^{(II)}$  &  $g_{1}+g_{7}$    &    15  &  1  &  $\widehat{E}_{15,1}$  &  $13g_{1}+g_{2}$    &    15  &  2  &  $\widehat{E}_{15,2}^{(I)}$  &  $6g_{1}+g_{3}$\\
15  &  2  &  $\widehat{E}_{15,2}^{(II)}$  &  $5g_{1}+2g_{2}$    &    15  &  3  &  $\widehat{E}_{15,3}^{(I)}$  &  $2g_{1}+g_{2}+g_{4}$    &    15  &  3  &  $\widehat{E}_{15,3}^{(II)}$  &  $3g_{1}+g_{5}$\\
15  &  4  &  $\widehat{E}_{15,4}^{(I)}$  &  $g_{1}+g_{8}$    &    15  &  4  &  $\widehat{E}_{15,4}^{(II)}$  &  $2g_{2}+g_{3}$    &    15  &  5  &  $\widehat{E}_{15,5}$  &  $g_{3}+g_{5}$\\
16  &  1  &  $\widehat{E}_{16,1}$  &  $14g_{1}+g_{2}$    &    16  &  2  &  $\widehat{E}_{16,2}$  &  $6g_{1}+g_{4}$    &    16  &  3  &  $\widehat{E}_{16,3}^{(I)}$  &  $3g_{1}+g_{6}$\\
16  &  3  &  $\widehat{E}_{16,3}^{(II)}$  &  $2g_{1}+3g_{2}$    &    16  &  4  &  $\widehat{E}_{16,4}^{(I)}$  &  $2g_{2}+g_{4}$    &    16  &  4  &  $\widehat{E}_{16,4}^{(II)}$  &  $g_{1}+g_{3}+g_{4}$\\
16  &  5  &  $\widehat{E}_{16,5}$  &  $g_{9}$    &    17  &  1  &  $\widehat{E}_{17,1}$  &  $15g_{1}+g_{2}$    &    17  &  2  &  $\widehat{E}_{17,2}^{(I)}$  &  $7g_{1}+g_{3}$\\
17  &  2  &  $\widehat{E}_{17,2}^{(II)}$  &  $6g_{1}+2g_{2}$    &    17  &  3  &  $\widehat{E}_{17,3}$  &  $3g_{1}+g_{2}+g_{3}$    &    17  &  4  &  $\widehat{E}_{17,4}^{(I)}$  &  $g_{1}+2g_{4}$\\
17  &  4  &  $\widehat{E}_{17,4}^{(II)}$  &  $g_{1}+g_{2}+g_{5}$    &    17  &  4  &  $\widehat{E}_{17,4}^{(III)}$  &  $2g_{1}+2g_{3}$    &    17  &  4  &  $\widehat{E}_{17,4}^{(IV)}$  &  $4g_{2}$\\
17  &  5  &  $\widehat{E}_{17,5}$  &  $g_{3}+g_{6}$    &    18  &  1  &  $\widehat{E}_{18,1}$  &  $16g_{1}+g_{2}$    &    18  &  2  &  $\widehat{E}_{18,2}$  &  $7g_{1}+g_{4}$\\
18  &  3  &  $\widehat{E}_{18,3}^{(I)}$  &  $3g_{1}+g_{2}+g_{4}$    &    18  &  3  &  $\widehat{E}_{18,3}^{(II)}$  &  $4g_{1}+g_{5}$    &    18  &  4  &  $\widehat{E}_{18,4}^{(I)}$  &  $g_{1}+g_{2}+g_{6}$\\
18  &  4  &  $\widehat{E}_{18,4}^{(II)}$  &  $2g_{1}+g_{7}$    &    18  &  5  &  $\widehat{E}_{18,5}^{(I)}$  &  $g_{2}+2g_{3}$    &    18  &  5  &  $\widehat{E}_{18,5}^{(II)}$  &  $g_{4}+g_{5}$\\
19  &  1  &  $\widehat{E}_{19,1}$  &  $17g_{1}+g_{2}$    &    19  &  2  &  $\widehat{E}_{19,2}^{(I)}$  &  $8g_{1}+g_{3}$    &    19  &  2  &  $\widehat{E}_{19,2}^{(II)}$  &  $7g_{1}+2g_{2}$\\
19  &  3  &  $\widehat{E}_{19,3}^{(I)}$  &  $4g_{1}+g_{6}$    &    19  &  3  &  $\widehat{E}_{19,3}^{(II)}$  &  $3g_{1}+3g_{2}$    &    19  &  4  &  $\widehat{E}_{19,4}^{(I)}$  &  $2g_{1}+g_{8}$\\
19  &  4  &  $\widehat{E}_{19,4}^{(II)}$  &  $g_{1}+2g_{2}+g_{3}$    &    19  &  5  &  $\widehat{E}_{19,5}^{(I)}$  &  $g_{4}+g_{6}$    &    19  &  5  &  $\widehat{E}_{19,5}^{(II)}$  &  $g_{2}+g_{7}$\\
19  &  6  &  $\widehat{E}_{19,6}$  &  $3g_{3}$    &    20  &  1  &  $\widehat{E}_{20,1}$  &  $18g_{1}+g_{2}$    &    20  &  2  &  $\widehat{E}_{20,2}$  &  $8g_{1}+g_{4}$\\
20  &  3  &  $\widehat{E}_{20,3}$  &  $4g_{1}+g_{2}+g_{3}$    &    20  &  4  &  $\widehat{E}_{20,4}^{(I)}$  &  $g_{1}+2g_{2}+g_{4}$    &    20  &  4  &  $\widehat{E}_{20,4}^{(II)}$  &  $2g_{1}+g_{3}+g_{4}$\\
20  &  5  &  $\widehat{E}_{20,5}^{(I)}$  &  $g_{2}+g_{8}$    &    20  &  5  &  $\widehat{E}_{20,5}^{(II)}$  &  $g_{1}+g_{3}+g_{5}$    &    21  &  1  &  $\widehat{E}_{21,1}$  &  $19g_{1}+g_{2}$\\
\hline
\end{tabular}
}
\caption{Expression of the entries of \cite[Appendix]{CDGK} in terms of the $g$ vectors (Part 1). Here the first column gives the genus $g$ and the second the value of $\phi$. The next two columns give the description of the polarization according to \cite{CDGK} and in terms of the $g_i$. }
\label{Expression_Isotropic_Vectors_Part1}
\end{center}
\end{table}

\begin{table}
\begin{center}
{\scriptsize
\begin{tabular}{|cccc|cccc|cccc|}
\hline
$g$  &  $\phi$  &   CDGK   &   DH     &    $g$  &  $\phi$  &  CDGK  &  DH   &    $g$  &   $\phi$  &  CDGK  &  DH\\
\hline
21  &  2  &  $\widehat{E}_{21,2}^{(I)}$  &  $9g_{1}+g_{3}$    &    21  &  2  &  $\widehat{E}_{21,2}^{(II)}$  &  $8g_{1}+2g_{2}$    &    21  &  3  &  $\widehat{E}_{21,3}^{(I)}$  &  $5g_{1}+g_{5}$\\
21  &  3  &  $\widehat{E}_{21,3}^{(II)}$  &  $4g_{1}+g_{2}+g_{4}$    &    21  &  4  &  $\widehat{E}_{21,4}^{(I)}$  &  $g_{1}+4g_{2}$    &    21  &  4  &  $\widehat{E}_{21,4}^{(II)}$  &  $3g_{1}+2g_{3}$\\
21  &  4  &  $\widehat{E}_{21,4}^{(III)}$  &  $2g_{1}+2g_{4}$    &    21  &  4  &  $\widehat{E}_{21,4}^{(IV)}$  &  $2g_{1}+g_{2}+g_{5}$    &    21  &  5  &  $\widehat{E}_{21,5}^{(I)}$  &  $g_{1}+g_{9}$\\
21  &  5  &  $\widehat{E}_{21,5}^{(II)}$  &  $g_{2}+g_{3}+g_{4}$    &    21  &  6  &  $\widehat{E}_{21,6}$  &  $2g_{5}$    &    22  &  1  &  $\widehat{E}_{22,1}$  &  $20g_{1}+g_{2}$\\
22  &  2  &  $\widehat{E}_{22,2}$  &  $9g_{1}+g_{4}$    &    22  &  3  &  $\widehat{E}_{22,3}^{(I)}$  &  $5g_{1}+g_{6}$    &    22  &  3  &  $\widehat{E}_{22,3}^{(II)}$  &  $4g_{1}+3g_{2}$\\
22  &  4  &  $\widehat{E}_{22,4}^{(I)}$  &  $2g_{1}+g_{2}+g_{6}$    &    22  &  4  &  $\widehat{E}_{22,4}^{(II)}$  &  $3g_{1}+g_{7}$    &    22  &  5  &  $\widehat{E}_{22,5}^{(I)}$  &  $g_{1}+g_{3}+g_{6}$\\
22  &  5  &  $\widehat{E}_{22,5}^{(II)}$  &  $2g_{2}+g_{5}$    &    22  &  5  &  $\widehat{E}_{22,5}^{(III)}$  &  $g_{2}+2g_{4}$    &    22  &  6  &  $\widehat{E}_{22,6}$  &  $g_{10}$\\
23  &  1  &  $\widehat{E}_{23,1}$  &  $21g_{1}+g_{2}$    &    23  &  2  &  $\widehat{E}_{23,2}^{(I)}$  &  $10g_{1}+g_{3}$    &    23  &  2  &  $\widehat{E}_{23,2}^{(II)}$  &  $9g_{1}+2g_{2}$\\
23  &  3  &  $\widehat{E}_{23,3}$  &  $5g_{1}+g_{2}+g_{3}$    &    23  &  4  &  $\widehat{E}_{23,4}^{(I)}$  &  $2g_{1}+2g_{2}+g_{3}$    &    23  &  4  &  $\widehat{E}_{23,4}^{(II)}$  &  $3g_{1}+g_{8}$\\
23  &  5  &  $\widehat{E}_{23,5}^{(I)}$  &  $g_{1}+g_{2}+2g_{3}$    &    23  &  5  &  $\widehat{E}_{23,5}^{(II)}$  &  $g_{1}+g_{4}+g_{5}$    &    23  &  5  &  $\widehat{E}_{23,5}^{(III)}$  &  $2g_{2}+g_{6}$\\
23  &  6  &  $\widehat{E}_{23,6}$  &  $g_{3}+g_{8}$    &    24  &  1  &  $\widehat{E}_{24,1}$  &  $22g_{1}+g_{2}$    &    24  &  2  &  $\widehat{E}_{24,2}$  &  $10g_{1}+g_{4}$\\
24  &  3  &  $\widehat{E}_{24,3}^{(I)}$  &  $6g_{1}+g_{5}$    &    24  &  3  &  $\widehat{E}_{24,3}^{(II)}$  &  $5g_{1}+g_{2}+g_{4}$    &    24  &  4  &  $\widehat{E}_{24,4}^{(I)}$  &  $2g_{1}+2g_{2}+g_{4}$\\
24  &  4  &  $\widehat{E}_{24,4}^{(II)}$  &  $3g_{1}+g_{3}+g_{4}$    &    24  &  5  &  $\widehat{E}_{24,5}^{(I)}$  &  $g_{1}+g_{4}+g_{6}$    &    24  &  5  &  $\widehat{E}_{24,5}^{(II)}$  &  $g_{1}+g_{2}+g_{7}$\\
24  &  5  &  $\widehat{E}_{24,5}^{(III)}$  &  $3g_{2}+g_{3}$    &    24  &  6  &  $\widehat{E}_{24,6}^{(I)}$  &  $2g_{3}+g_{4}$    &    24  &  6  &  $\widehat{E}_{24,6}^{(II)}$  &  $g_{5}+g_{6}$\\
25  &  1  &  $\widehat{E}_{25,1}$  &  $23g_{1}+g_{2}$    &    25  &  2  &  $\widehat{E}_{25,2}^{(I)}$  &  $11g_{1}+g_{3}$    &    25  &  2  &  $\widehat{E}_{25,2}^{(II)}$  &  $10g_{1}+2g_{2}$\\
25  &  3  &  $\widehat{E}_{25,3}^{(I)}$  &  $6g_{1}+g_{6}$    &    25  &  3  &  $\widehat{E}_{25,3}^{(II)}$  &  $5g_{1}+3g_{2}$    &    25  &  4  &  $\widehat{E}_{25,4}^{(I)}$  &  $2g_{1}+4g_{2}$\\
25  &  4  &  $\widehat{E}_{25,4}^{(II)}$  &  $4g_{1}+2g_{3}$    &    25  &  4  &  $\widehat{E}_{25,4}^{(III)}$  &  $3g_{1}+2g_{4}$    &    25  &  4  &  $\widehat{E}_{25,4}^{(IV)}$  &  $3g_{1}+g_{2}+g_{5}$\\
25  &  5  &  $\widehat{E}_{25,5}^{(I)}$  &  $2g_{1}+g_{3}+g_{5}$    &    25  &  5  &  $\widehat{E}_{25,5}^{(II)}$  &  $g_{1}+g_{2}+g_{8}$    &    25  &  5  &  $\widehat{E}_{25,5}^{(III)}$  &  $3g_{2}+g_{4}$\\
25  &  6  &  $\widehat{E}_{25,6}^{(I)}$  &  $g_{1}+3g_{3}$    &    25  &  6  &  $\widehat{E}_{25,6}^{(II)}$  &  $2g_{6}$    &    25  &  6  &  $\widehat{E}_{25,6}^{(III)}$  &  $g_{4}+g_{7}$\\
26  &  1  &  $\widehat{E}_{26,1}$  &  $24g_{1}+g_{2}$    &    26  &  2  &  $\widehat{E}_{26,2}$  &  $11g_{1}+g_{4}$    &    26  &  3  &  $\widehat{E}_{26,3}$  &  $6g_{1}+g_{2}+g_{3}$\\
26  &  4  &  $\widehat{E}_{26,4}^{(I)}$  &  $3g_{1}+g_{2}+g_{6}$    &    26  &  4  &  $\widehat{E}_{26,4}^{(II)}$  &  $4g_{1}+g_{7}$    &    26  &  5  &  $\widehat{E}_{26,5}^{(I)}$  &  $2g_{1}+g_{9}$\\
26  &  5  &  $\widehat{E}_{26,5}^{(II)}$  &  $g_{1}+g_{2}+g_{3}+g_{4}$    &    26  &  5  &  $\widehat{E}_{26,5}^{(III)}$  &  $5g_{2}$    &    26  &  6  &  $\widehat{E}_{26,6}^{(I)}$  &  $g_{2}+g_{3}+g_{5}$\\
26  &  6  &  $\widehat{E}_{26,6}^{(II)}$  &  $g_{4}+g_{8}$    &    27  &  1  &  $\widehat{E}_{27,1}$  &  $25g_{1}+g_{2}$    &    27  &  2  &  $\widehat{E}_{27,2}^{(I)}$  &  $12g_{1}+g_{3}$\\
27  &  2  &  $\widehat{E}_{27,2}^{(II)}$  &  $11g_{1}+2g_{2}$    &    27  &  3  &  $\widehat{E}_{27,3}^{(I)}$  &  $7g_{1}+g_{5}$    &    27  &  3  &  $\widehat{E}_{27,3}^{(II)}$  &  $6g_{1}+g_{2}+g_{4}$\\
27  &  4  &  $\widehat{E}_{27,4}^{(I)}$  &  $3g_{1}+2g_{2}+g_{3}$    &    27  &  4  &  $\widehat{E}_{27,4}^{(II)}$  &  $4g_{1}+g_{8}$    &    27  &  5  &  $\widehat{E}_{27,5}^{(I)}$  &  $2g_{1}+g_{3}+g_{6}$\\
27  &  5  &  $\widehat{E}_{27,5}^{(II)}$  &  $g_{1}+2g_{2}+g_{5}$    &    27  &  5  &  $\widehat{E}_{27,5}^{(III)}$  &  $g_{1}+g_{2}+2g_{4}$    &    27  &  6  &  $\widehat{E}_{27,6}^{(I)}$  &  $g_{1}+2g_{5}$\\
27  &  6  &  $\widehat{E}_{27,6}^{(II)}$  &  $g_{3}+2g_{4}$    &    27  &  6  &  $\widehat{E}_{27,6}^{(III)}$  &  $g_{2}+g_{9}$    &    28  &  1  &  $\widehat{E}_{28,1}$  &  $26g_{1}+g_{2}$\\
28  &  2  &  $\widehat{E}_{28,2}$  &  $12g_{1}+g_{4}$    &    28  &  3  &  $\widehat{E}_{28,3}^{(I)}$  &  $7g_{1}+g_{6}$    &    28  &  3  &  $\widehat{E}_{28,3}^{(II)}$  &  $6g_{1}+3g_{2}$\\
28  &  4  &  $\widehat{E}_{28,4}^{(I)}$  &  $3g_{1}+2g_{2}+g_{4}$    &    28  &  4  &  $\widehat{E}_{28,4}^{(II)}$  &  $4g_{1}+g_{3}+g_{4}$    &    28  &  5  &  $\widehat{E}_{28,5}^{(I)}$  &  $2g_{1}+g_{2}+2g_{3}$\\
28  &  5  &  $\widehat{E}_{28,5}^{(II)}$  &  $2g_{1}+g_{4}+g_{5}$    &    28  &  5  &  $\widehat{E}_{28,5}^{(III)}$  &  $g_{1}+2g_{2}+g_{6}$    &    28  &  6  &  $\widehat{E}_{28,6}^{(I)}$  &  $g_{1}+g_{10}$\\
28  &  6  &  $\widehat{E}_{28,6}^{(II)}$  &  $3g_{4}$    &    28  &  6  &  $\widehat{E}_{28,6}^{(III)}$  &  $g_{2}+g_{3}+g_{6}$    &    28  &  7  &  $\widehat{E}_{28,7}$  &  $2g_{3}+g_{5}$\\
29  &  1  &  $\widehat{E}_{29,1}$  &  $27g_{1}+g_{2}$    &    29  &  2  &  $\widehat{E}_{29,2}^{(I)}$  &  $13g_{1}+g_{3}$    &    29  &  2  &  $\widehat{E}_{29,2}^{(II)}$  &  $12g_{1}+2g_{2}$\\
29  &  3  &  $\widehat{E}_{29,3}$  &  $7g_{1}+g_{2}+g_{3}$    &    29  &  4  &  $\widehat{E}_{29,4}^{(I)}$  &  $3g_{1}+4g_{2}$    &    29  &  4  &  $\widehat{E}_{29,4}^{(II)}$  &  $5g_{1}+2g_{3}$\\
29  &  4  &  $\widehat{E}_{29,4}^{(III)}$  &  $4g_{1}+2g_{4}$    &    29  &  4  &  $\widehat{E}_{29,4}^{(IV)}$  &  $4g_{1}+g_{2}+g_{5}$    &    29  &  5  &  $\widehat{E}_{29,5}^{(I)}$  &  $2g_{1}+g_{4}+g_{6}$\\
29  &  5  &  $\widehat{E}_{29,5}^{(II)}$  &  $2g_{1}+g_{2}+g_{7}$    &    29  &  5  &  $\widehat{E}_{29,5}^{(III)}$  &  $g_{1}+3g_{2}+g_{3}$    &    29  &  6  &  $\widehat{E}_{29,6}^{(I)}$  &  $g_{1}+g_{3}+g_{8}$\\
29  &  6  &  $\widehat{E}_{29,6}^{(II)}$  &  $2g_{2}+2g_{3}$    &    29  &  6  &  $\widehat{E}_{29,6}^{(III)}$  &  $g_{2}+g_{4}+g_{5}$    &    30  &  1  &  $\widehat{E}_{30,1}$  &  $28g_{1}+g_{2}$\\
30  &  2  &  $\widehat{E}_{30,2}$  &  $13g_{1}+g_{4}$    &    30  &  3  &  $\widehat{E}_{30,3}^{(I)}$  &  $7g_{1}+g_{2}+g_{4}$    &    30  &  3  &  $\widehat{E}_{30,3}^{(II)}$  &  $8g_{1}+g_{5}$\\
30  &  4  &  $\widehat{E}_{30,4}^{(I)}$  &  $4g_{1}+g_{2}+g_{6}$    &    30  &  4  &  $\widehat{E}_{30,4}^{(II)}$  &  $5g_{1}+g_{7}$    &    30  &  5  &  $\widehat{E}_{30,5}^{(I)}$  &  $3g_{1}+g_{3}+g_{5}$\\
30  &  5  &  $\widehat{E}_{30,5}^{(II)}$  &  $2g_{1}+g_{2}+g_{8}$    &    30  &  5  &  $\widehat{E}_{30,5}^{(III)}$  &  $g_{1}+3g_{2}+g_{4}$    &    30  &  6  &  $\widehat{E}_{30,6}^{(I)}$  &  $g_{1}+2g_{3}+g_{4}$\\
30  &  6  &  $\widehat{E}_{30,6}^{(II)}$  &  $g_{1}+g_{5}+g_{6}$    &    30  &  6  &  $\widehat{E}_{30,6}^{(III)}$  &  $g_{2}+g_{4}+g_{6}$    &    30  &  6  &  $\widehat{E}_{30,6}^{(IV)}$  &  $2g_{2}+g_{7}$\\
30  &  7  &  $\widehat{E}_{30,7}^{(I)}$  &  $g_{3}+g_{9}$    &    30  &  7  &  $\widehat{E}_{30,7}^{(II)}$  &  $2g_{1}+g_{2}+g_{8}$    &      &&&             \\
\hline
\end{tabular}
}
\caption{Expression of the entries of \cite[Appendix]{CDGK} in terms of the $g$ vectors (Part 2)}
\label{Expression_Isotropic_Vectors_Part2}
\end{center}
\end{table}

\newpage
\subsection{The Tits building}\label{subsec:isotropic}

As we already recalled the $0$- and $1$-dimensional cusps of the modular varieties ${\cM}_{\En,h}= \Gamma^+_h \backslash \cD_N$
correspond to the orbits of isotropic vectors $l$ and isotropic planes $h$ in the lattice $N= U + U(2) + E_8(-2)$ with respect to the groups $\Gamma^+_h$.
The inclusions $l \subset h$ characterize when a $0$-dimensional cusp is contained in a $1$-dimensional cusp. Taking orbits modulo the group $\Gamma^+_h$
defines the Tits building which thus incorporates the entire combinatorial structure of the boundary.
We will now investigate this systematically for small degrees.

The classical case of unpolarized Enriques surfaces is well known. We start by recalling that $\cM_{\En}= \Orth^+(N) \backslash \cD_N$ has two $0$-dimensional and two $1$-dimensional cusps each.
Proofs of this were given by Sterk \cite[Propositions 4.5 and 4.6]{Ste} and Allcock \cite[Corollary 4]{All}.
By applying the computational techniques of Section \ref{Sec_Algorithms} we confirm these results.
It is not difficult to give explicit representatives of these orbits.
The orbits of isotropic lines are spanned by $L_1 = \ZZ e_1$ and $L_2 = \ZZ e_3$.
The orbits of isotropic planes of $N$ are $P_1 = \ZZ e_1 + \ZZ e_3$ and $P_2 = \ZZ  (2 e_1 + 2 e_2 + w) + \ZZ e_3$ with $w$
a vector of norm $4$ in $E_8$, which then, viewed as a vector in $E_8(-2)$ has norm $w^2=-8$ in $E_8(-2)$.
Here $(e_1,e_2)$ and $(e_3,e_4)$ are standard bases of $U$ and $U(2)$.
The Tits building is displayed in Table \ref{table_subgroups1}.
The stabilizer of $L_1$ has an image in the discriminant group which is equal to the full group, while
the image of the stabilizer of $L_2$ has index $527 = 17 \cdot 31$.
The stabilizer of the plane $P_1$ has an image of index $527$ while the image of the stabilizer of $P_2$ has an image of index $23715 = 3^2 \cdot 5 \cdot 17 \cdot 31$.

Here we would also like to mention that the space $\widetilde {\cM}_{\En}= \widetilde \Orth^+(N)\backslash \cD_N$ has $528$ $0$-dimensional cusps (corresponding to isotropic lines) and $24242$
$1$-dimensional cusps (corresponding isotropic planes). Under the group $\Orth^+(\FF_2)$ the $528$ $0$-dimensional cusps decompose into two orbits, one of length $527$ and one of length $1$, see
the above discussion and \cite[p. 534]{CDL}. The set of $24242$ isotropic planes also has $2$ orbits, and these are of length $527$ and $23715$ respectively.

When working with the group $\Gamma_h$, we need to compute the orbits for a subgroup of the full
isometry group $G$. That is, given an orbit $xG$ we write $G_x$ for the stabilizer of $x$ by $G$.
The decomposition $xG = \cup_{i\in I} x_i \Gamma_h$ corresponds to a double coset decomposition
\begin{equation*}
G = \cup_{i\in I} G_x h_i \Gamma_h \mbox{~for~} x_i = x h_i.
\end{equation*}
Thus orbit splitting can be done with double coset decomposition.
This is in general a difficult problem but in the case of finite groups there are well known
algorithms \cite[Sec. 8.1.1]{Seress}.

\begin{lemma}
Let $G$ be a group, $U$ a normal subgroup and $K$, $H$ two subgroups of $G$ such that $U \subset K$.
Then the quotient map $G \to  G/U$ establishes a many-to-one correspondence between
double coset decompositions of $G$ by $(K,H)$ and of $G / U$ by $(K/U, \overline{H})$ where $\overline{H}$
is the image of $H$ in $G/U$.
\end{lemma}
\begin{proof}
Let us take a double coset decomposition
\begin{equation*}
G = \cup_{i\in I} K g_i H
\end{equation*}
Then mapping to the quotient we obtain
\begin{equation*}
G/U = \cup_{i\in I} K/U \overline{g_i} \overline{H}.
\end{equation*}
Since $U \subset K$, the double coset $K g_i H$ is actually a union of left $U$ cosets. Therefore the double cosets
$C_i = K/U \overline{g_i} \overline{H}$ are disjoint, i.e., $C_i\cap C_j = \emptyset$ if $i\not= j$.
This establishes that the mapping is well defined and it is surjective by construction.
\end{proof}

We can apply the above lemma to our case with $U=\widetilde\Orth(N)$ the kernel of the action of $\Orth(N)$ on the discriminant
$K = \Gamma_h$ and $H = G_x$.
The second key ingredient is that the quotient $\Orth(N) / \widetilde\Orth(N) \cong \Orth^+(\FF_2)$ is finite. We can apply the
existing approach for finite groups, as implemented in \cite{gap}, and thus reduce the orbit splitting from $\Orth(N)$ to the
group $\Gamma_h$.

This approach allows us to compute the number of orbits of lines, planes and flags, and thus the Tits building, for each subgroup $\Gamma_h$.
The obtained data on the orbit splitting, i.e. the number of isotropic lines and planes as well as inclusions are given in Tables \ref{table_subgroups1} and \ref{table_subgroups2}
where they are labeled by $\# I_1$, $\# I_2$ and $\# I_{12}$ respectively.
The pictures of the first $8$ coset graphs are given in Figure \ref{FirstEightGraphs}.
We note that Case 1 coincides with the Tits building found in \cite[Fig 14]{Ste}.

\begin{figure}
\centering
\begin{minipage}[b]{4.1cm}
\centering
\epsfig{figure=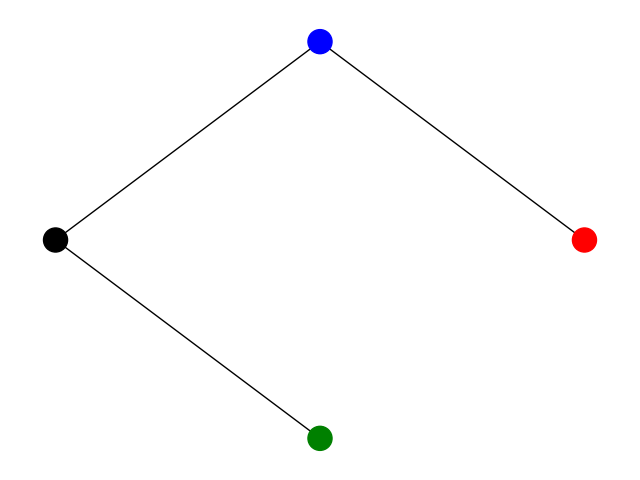,width=4.1cm}
Full group
\end{minipage}
\begin{minipage}[b]{4.1cm}
\centering
\epsfig{figure=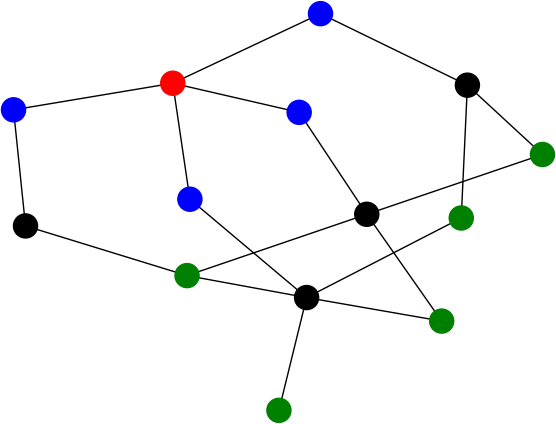,width=4.1cm}
Case 1
\end{minipage}
\begin{minipage}[b]{4.1cm}
\centering
\epsfig{figure=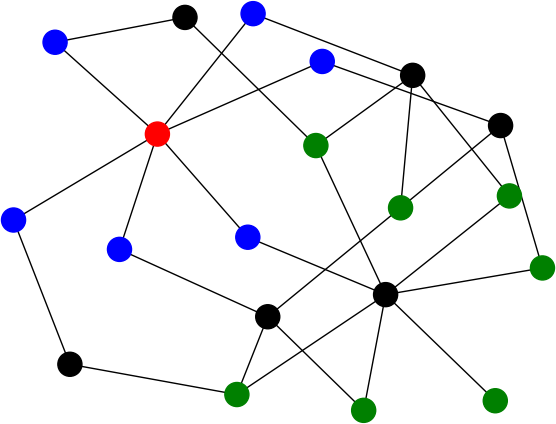,width=4.1cm}
Case 2
\end{minipage}
\begin{minipage}[b]{4.1cm}
\centering
\epsfig{figure=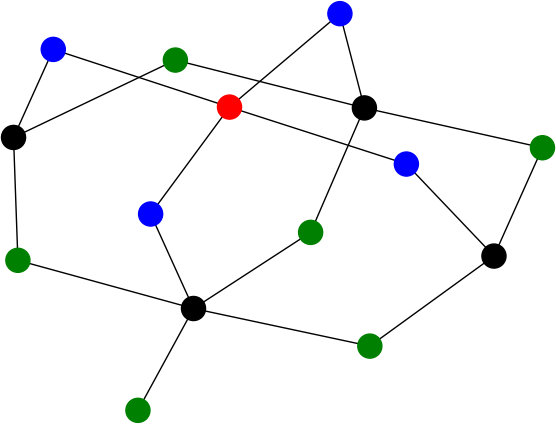,width=4.1cm}
Case 3
\end{minipage}
\begin{minipage}[b]{4.1cm}
\centering
\epsfig{figure=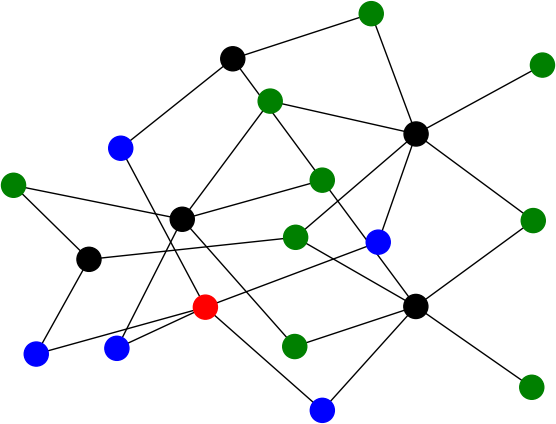,width=4.1cm}
Case 4
\end{minipage}
\begin{minipage}[b]{4.1cm}
\centering
\epsfig{figure=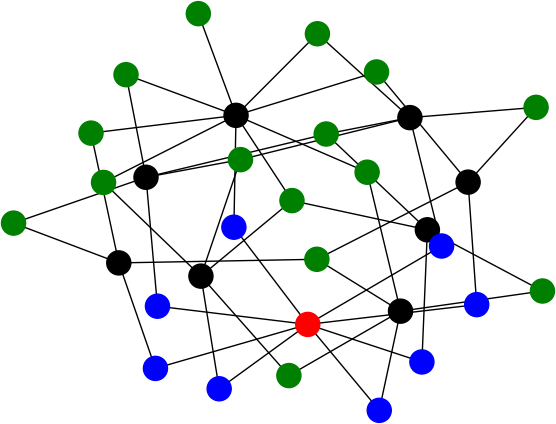,width=4.1cm}
Case 5
\end{minipage}
\begin{minipage}[b]{4.1cm}
\centering
\epsfig{figure=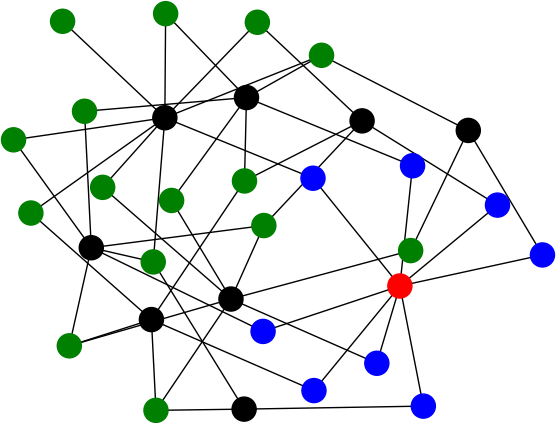,width=4.1cm}
Case 6
\end{minipage}
\begin{minipage}[b]{4.1cm}
\centering
\epsfig{figure=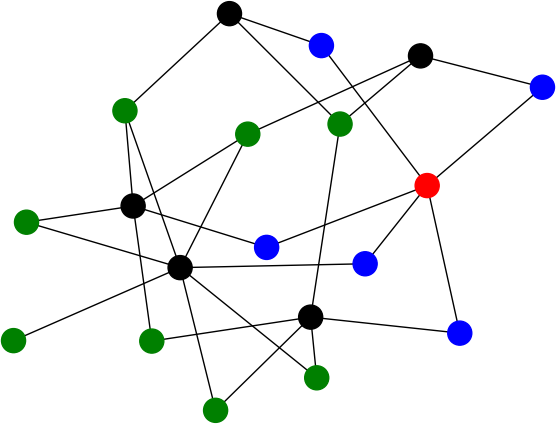,width=4.1cm}
Case 7
\end{minipage}
\begin{minipage}[b]{4.1cm}
\centering
\epsfig{figure=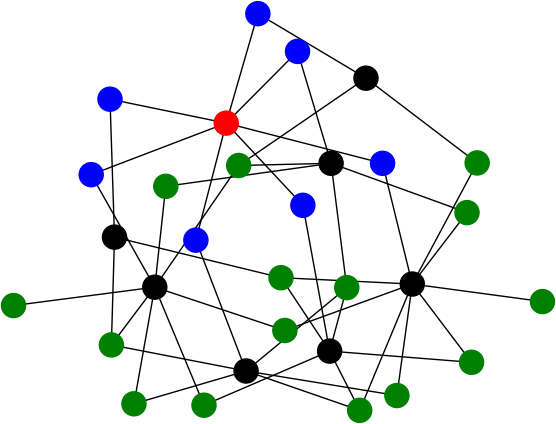,width=4.1cm}
Case 8
\end{minipage}

\caption{Coset graphs of the first $8$ coset graphs from Table \ref{table_subgroups1}.
The red, blue, black and green dots correspond to the orbit arising from the orbit of $L_1$, $P_1$, $L_2$ and $P_2$.}
\label{FirstEightGraphs}
\end{figure}

\subsection{Classical cases}\label{subsec:classical}
Here we briefly discuss how our computations fit in with some classical results.

\subsubsection{The degree $2$ case}
There is only one degree $2$ polarization $h$, namely $h = g_2$. This linear system is not base point free,
which follows e.g. because $\phi=1$, i.e. there is an elliptic curve on which $h$ has degree $1$. The linear system $|2h|$, however, is base-point free
and gives rise to what is classically known a double plane representation. More precisely, $|2h|$ maps a general
Enriques surface $2:1$ onto a del Pezzo surface in $\PP^4$ which is the intersection of two rank $3$ quadrics, see \cite[Section 3.3]{BaPe}, \cite[Section 3.5]{CDL}. By
\cite[Theorem 3.9]{BaPe} a generic Enriques surface admits  $2^7\cdot 7 \cdot 31$ different double plane representations. This implies that $[\Orth^+(N):\Gamma_h]=2^7\cdot 7 \cdot 31$ or alternatively,
that
\begin{equation*}
[\Gamma_h: \tilde\Orth^+(N)]= 2^{21} \cdot 3^5 \cdot 5^2 \cdot 7 \cdot 17 \cdot 31/ 2^7 \cdot 7 \cdot 31= 2^{14} \cdot 3^5 \cdot 5^2 \cdot 7 \cdot 31.
\end{equation*}
This is Case 1 in Table \ref{table_subgroups1}.

We note that this case was also treated by Sterk \cite{Ste} who referred to degree $2$ polarizations as almost polarizations. There it is  also
proved, see \cite[Section 4.4]{Ste}, that the corresponding modular variety has five $0$-dimensional and nine $1$-dimensional cusps, in
agreement with our results in Table \ref{table_subgroups1}. As we already mentioned above, the Tits building computed by Sterk \cite[Fig 14]{Ste} agrees with our graph for Case 1.

\subsubsection{Enriques realizations}

In degree $6$ we have two polarizations, which are distinguished by the $\phi$-invariant which can be $1$ or $2$. These are given by $g_1 + g_2$ and $g_3$ respectively.
In the first case the linear system is not base-point free, in the other case it defines, for a generic Enriques surface, a birational map onto a non-normal sextic surface
in $\PP^3$ with double locus along the edges of a tetrahedron, see \cite[Section 3.1]{BaPe}, \cite[Section 3.5]{CDL}. This is historically the first realization of an Enriques surface.
By \cite[Theorem 3.10]{BaPe} a general Enriques surface $S$
admits $2^{11} \cdot 5 \cdot 17 \cdot 31$ such realizations. Note, however, that $h$ and $h+K_S$ define projectively inequivalent models. For us this means that the morphism
$\cM_{\En,h} \to {\cM}_{\En}$ has degree $2^{10} \cdot 5 \cdot 17 \cdot 31$ and hence
\begin{equation*}
[\Gamma_h: \tilde\Orth^+(N)]= 2^{21} \cdot 3^5 \cdot 5^2 \cdot 7 \cdot 17 \cdot 31 / 2^{10} \cdot 5 \cdot 17 \cdot 31= 2^{11} \cdot 3^5 \cdot 5 \cdot 7
\end{equation*}
agreeing with Case 4 in Table \ref{table_subgroups1}.

\subsubsection{Reye congruences}

It follows from Table \ref{ListVectorsNormAtMost30} that we have $3$ different polarizations in degree $10$, namely $4g_1 + g_2$, $g_1 + g_4$ and $g_5$.
According to Table \ref{Expression_Isotropic_Vectors_Part1}  these have $\phi$-invariants $1$, $2$ and $3$ respectively. In the first case
the linear system $|h|$ is not base-point free, in the second it cannot be ample. In the third case the linear system $|h|$ defines an embedding for the general Enriques surface $S$ and thus a degree $10$ model in $\PP^5$,
see also \cite[Section 3]{BaPe}, \cite[Section 3.5]{CDL}.
This is classically known as a Reye congruence, or, according to \cite{CDL} as a Fano model.  By  \cite[Theorem 3.11]{BaPe} a general Enriques surface admits $2^{14} \cdot 3 \cdot 17 \cdot 31$ inequivalent
representations as a degree $10$ surface in $\PP^5$. Since $|h|$ and $|h+K_S|$ define different models we can conclude that  the morphism
$\cM_{\En,h} \to {\cM}_{\En}$ has degree $2^{13} \cdot 3 \cdot 17 \cdot 31$ and hence
\begin{equation*}
[\Gamma_h: \tilde\Orth^+(N)]= 2^{21} \cdot 3^5 \cdot 5^2 \cdot 7 \cdot 17 \cdot 31 / 2^{13} \cdot 3 \cdot 17 \cdot 31= 2^{8} \cdot 3^4 \cdot 5^2 \cdot 7
\end{equation*}
which agrees exactly with Case 7 in Table \ref{table_subgroups1}.

\section{Algorithms for working with indefinite forms}\label{Sec_Algorithms}

The methods used in this work are, we believe, of wider interest and thus we explain in this section
in some detail
how they work. The code is available via \cite{Polyhedral} in both a GAP and C++ version.
The emphasis is on practical techniques.
In this section, a lattice is a free $\ZZ$-module of rank $n$, which will often be identified with $\ZZ^n$,
equipped with an integer quadratic form $A$ which can possibly be degenerate.

The first class of problems are related to groups. That is, given an integer quadratic form $A$, we want to
compute a generating set of the integral automorphism group $\Orth(A)$. Next, we want to test the equivalence
of two integral quadratic forms $A_1$ and $A_2$ by an integral transformation and, if such an isomorphism exists,
produce it explicitly.

The second class of problems considered concerns vector representations.
That is, given an integer quadratic form $A$ and an integer $\beta\not= 0$, we ask to find all orbit representatives of
solutions $x$ of the equation $A[x] := x A x^T = \beta$. For $\beta = 0$ we are looking for primitive solutions.
We are also interested in finding $k$-planes of totally isotropic vectors.

As it turns out, both classes are closely related in our algorithmic approach. In Subsection \ref{subsection_int_group_algo}
we explain the group techniques used.  We then discuss the case of positive and hyperbolic lattices, for which
well known algorithms exist, in Subsection \ref{subsection_posdef_hyperbolic}.
In Subsection \ref{subsection_approximatemodels} we introduce the notion of approximate model of a lattice and,
finally, we show in Subsection \ref{subsection_solutionsproblems} how all techniques
together allow us to solve the above problems.

\subsection{Integral group algorithms}\label{subsection_int_group_algo}
The matrix groups $G \subset \GL_n(\QQ)$ that we will consider will be in general infinite
and will preserve a rank $n$ lattice $L\subset \ZZ^n$. In particular, this implies that $G\cap \GL_n(\ZZ)$
is a finite index subgroup in $G$.
We will need an algorithmic solution for the following problems:
\begin{itemize}
\item[{\bf Alg 1}] Compute a generating set of the intersection $G \cap \GL_n(\ZZ)$.
\item[{\bf Alg 2}] For $x \in \GL_n(\QQ)$ decide whether there is some $g\in G$ such that $g x\in \GL_n(\ZZ)$ and compute one such $g$.
\item[{\bf Alg 3}] Compute the right cosets of $G\cap \GL_n(\ZZ)$ in $G$.
\end{itemize}
Without the condition that a lattice $L$ is preserved by $G$ there is no reason to think there is a general algorithm as the groups are just too wild.
We will limit our exposition to {\bf Alg 1}. The other algorithms use the same ideas and are suitable adaptations
to the relevant context.

Let us take $L$ an integral rank $n$ lattice invariant under $G$ and denote by $L'$ the lattice $\ZZ^n$.
Obviously, we have $L\subset L'$ and there exists an integer $d>0$ such that $L \subset L' \subset L / d$.
When expressed in a basis of $L$ the group $G$ becomes an integral subgroup of $\GL_n(\ZZ)$.
By quotienting by $dL$ we obtain a map $\phi: G \mapsto \GL_n(\ZZ / d \ZZ)$ mapping the lattice $L'$ to a subset $S$
 of $(\ZZ /d \ZZ)^n$
 and the problem can be rephrased as first finding the stabilizer of $S$ under $\Im\,\phi$ and then computing
its pre-image in $G$.

The group $\GL_n(\ZZ/d\ZZ)$ is a finite group and finding set-stabilizers is a well
known problem with efficient algorithms \cite{leon1,leon2,jefferson}.
To find the pre-image of a group, the natural way is to use  Schreier's lemma
\cite[Lemma 4.2.1]{Seress}.
If the group $G$ is finite and has a faithful permutation representation
on a set $W$, then we can amalgamate the set-stabilizer and pre-image operations
in just one set-stabilizer
operation on a finite permutation group acting on $|W| + d^n$ points.

Because of its practical importance, it is essential to accelerate this algorithm
as much as possible. A possible speed-up is to use the factorization of the divisor $d$
into prime factors as $d = p_1\dots p_r$ and to iterate the computation prime by prime,
starting by the smallest occurring prime.
Another speed-up is not to consider the full set $(\ZZ / d \ZZ)^n$ of vectors, but instead
to select a vector $x$ in $S$ whose orbit $O_x$ is not contained in $S$. Then we compute
 the stabilizer for $S' = O_x\cap S$.
In this way $d^n$ is reduced to $\left\vert O_x \right\vert$ which is much smaller. Of course,
some additional iterations may be needed for the stabilizer to be computed since there could
be other vectors in $S$ whose orbit is contained in $S$.
If we know that a filtration is preserved by $G$, then it is good to start
the search of such $x$ in the smallest subspaces. The commonality between all these
approaches is that they replace a big computation with a dominating term $d^n$ into
smaller computations though at the expense of having many.
This algorithm is an evolution of the last one of \cite[Section 3.1]{BDPRS} where the
problem of finding the group of integral symmetries of a polytope was considered.
The GAP and the independent C++ version of the code are available at \cite{Polyhedral}.

\subsection{Positive definite and hyperbolic forms}\label{subsection_posdef_hyperbolic}
For positive definite quadratic forms there are well known methods \cite{plesken} for the equivalence
problem and for computing a generating set.
For the problem of finding representative solutions of $A[x] = \beta$, we can use the
Fincke-Pohst algorithm (cf. \cite[Algorithm~2.7.7]{cohen}).

For the case of hyperbolic lattices this becomes more involved, but is still doable using
the method of perfect forms. This is an inefficient technique, but it has the advantage that there are no limitations
regarding its use.
In this work we have used the Coxeter group structure for the lattice $U + E_8(-1)$.
This is possible because it is a reflective lattice, but most lattices do not have this property and so the perfect form method has to be used.

The enumeration of perfect forms is done via a variant of Mertens' algorithm \cite{Mertens}.
The main changes are an improvement in the way the facets of the perfect domain are enumerated
up to symmetry (see \cite{highlysymmetric} for a description of the algorithm and
\cite{Polyhedral} for implementations) and the use of the method of \ref{subsection_int_group_algo}
for finding automorphisms and testing isomorphisms of perfect domains.

\subsection{Approximate models and the case of signature $p,q\geq 2$}\label{subsection_approximatemodels}

Having dealt with definite forms and hyperbolic lattices, we now turn to signature $(p,q)$ with $p,q \geq 2$.
The definition below provides the main tool for our work:

\begin{definition}
Given an integral lattice $L$, an {\em approximate model} is defined by:
\begin{itemize}
\item A set of generators $\{g_1, \dots, g_m\}$ of a subgroup $\Ap(L)$
of $\Orth(L)$ named {\em approximate subgroup}.
\item An {\em oracle function} $\Ap(L,\beta)$ that, given a $\beta \not=0$, returns
a finite list $v_1, \dots, v_{k(\beta)}$ such that any vector of norm $\beta$
is equivalent by $\Ap(L)$ to one of the $v_i$.
For $\beta = 0$ the oracle function returns a list of primitive vectors of norm $0$ such that
any primitive vector of norm $0$ is equivalent to one such vector by an element of $\Ap(L)$.
\end{itemize}
\end{definition}
It is important to note that a lattice can potentially have an infinite number of approximate models
and that we do not claim that every lattice has an approximate model. The approximate subgroup is
a finite index subgroup of $\Orth(L)$ in all cases considered, here but we do not know if that
is always the case and we do not use this property here.

\begin{lemma}\label{sublattice_approx_model}
If $L$ is an integral non-degenerate even lattice, then $U + U + L$ has an approximate model.
\end{lemma}
\begin{proof}
Eichler's criterion \cite[\S 3]{eichler}
applies to this class of lattices
and provides an algorithm for obtaining the approximate orbit representatives.
We need to prove that we have a finite set of generators of a suitable
approximate subgroup. In this case we can take a suitable subgroup of the group $\Orth(U+U)$ together with the Eichler transvections.
In \cite[Example 3.7.2]{scattone} the lattice $U+U$ is identified with the determinant form on $M_{2,2}(\ZZ)$.
Thus $\SL_2(\ZZ)$ has an action on the left and an action on the right on $M_{2,2}(\ZZ)$.
In particular, $\SL_2(\ZZ) \times \SL_2(\ZZ)$ is a subgroup of $\Orth(U + U)$.
Since $\SL_2(\ZZ)$ is generated by $2$ elements this subgroup is generated by $4$ generators.

The second step of Eichler's criterion is to apply the Eichler transvections
$E_{e_i, x}$ (see \cite[Section 3.7]{scattone})
for $e_i$ with $1\leq i\leq 4$ one of the $4$ canonical isotropic vectors coming from the two hyperbolic planes $U$
and $x$ a vector orthogonal to $e_i$.
The Eichler transvections satisfy $E_{e, x}E_{e,y} = E_{e, x+y}$ for $e$ isotropic
and $x,y$ orthogonal to $e$.
According to \cite[Proposition 3.7.3]{scattone} we simply need the generators of $\SL_2(\ZZ)^2$ and the transvections $E_{e_i,v_{i,j}}$
with $1\leq i\leq 4$, $1\leq j\leq n-1$ and $(v_{i,j})_{1\leq j\leq n-1}$ forming a $\ZZ$-basis of $e_i^{\perp}$.
Thus, if the dimension of $U+U+L$ is $n$, we need $4$ generators from $\SL_2(\ZZ)^2$ and $4(n-1)$ from the transvections and so $4n$
together.
The proof of Proposition 3.7.3 in \cite{scattone} provides an explicit way of computing a set of possible vector representatives
and so the oracle function.
\end{proof}

It is important to note that the approximate model provided by the above lemma can be improved significantly in some cases.
The group provided by the Eichler algorithm acts trivially on the discriminant.
For a case such as $U + U + E_8(-2)$ this gets us $2^8$ orbit representatives. By adding
the isometries of the $E_8$ component to the approximate subgroup, we are reduced to just
$3$ representatives which is far better for computational purposes. This is because $W(E_8)$
has $3$ orbits in its action on $E_8 / 2 E_8$, their sizes being $1$, $120$ and $135$.

\begin{theorem}
Suppose $L'$ and $L$ are two integral lattices of rank $n$ with $L' \subset L$ and we have an
approximate model for $L$. Then we have an approximate model for $L'$.
\end{theorem}
\begin{proof}
We can compute the stabilizer $S$ of $L'$ under $\Ap(L)$ by {\bf Alg 1} and this gets us an approximate subgroup $\Ap(L')$.
By using {\bf Alg 3} we compute the right coset decomposition of $\Ap(L)$ under $S$ with
coset representatives $g_1$, \dots, $g_m$.
For $\beta\in \ZZ$,
the approximate model of $L$ gives us representatives $x_1$, \dots, $x_t$ of the orbits of vectors of norm $\beta$.
We then considers all the elements of the form $g_j x_i$ and keep the ones that are contained in $L'$.
This gets us our approximate orbit representatives $\Ap(L',\beta)$.
\end{proof}

In particular, the above shows that any lattice $c U + dU + W$ with $c,d\in \NN$ and $W$ integral and even
has an approximate model
via the following embedding in $U + U + W$:
\begin{equation*}
(x_1, x_2, y_1, y_2, w) \mapsto (c x_1, x_2, d y_1, y_2, w).
\end{equation*}

In fact much more is true:
\begin{theorem}\label{Existence_ApproximateModel}
Let $L$ be an integral lattice of signature $p,q\geq 2$ of dimension at least $7$.
Then $L$ has an approximate model.
\end{theorem}
\begin{proof}
Let us take the dual $L^{\vee}$. Since integral indefinite lattices of dimension at least $5$ have isotropic vectors (see \cite{meyer})
there is an isotropic vector $v_1$ in $L^{\vee}$.
Let us take a vector $g$ not orthogonal to $v_1$. Then the vector $v_2 = 2 (g.v_1) g - (g.g) v_1$ is
also isotropic and not orthogonal to $v_1$.

We then iterate this operation on $L^{\vee} \cap (\ZZ v_1 + \ZZ v_2)^{\perp}$, where this notation indicates that the orthogonal complement is taken in $L^{\vee}$,
and find two isotropic vectors
$v_3$, $v_4$. We define $K = L^{\vee} \cap (\ZZ v_1 + \ZZ v_2 + \ZZ v_3 + \ZZ v_4)^{\perp}$ and
taking the dual we  obtain
\begin{equation*}
L \subset U(c) + U(d) + K^{\vee}   \mbox{~for~some~} c,d\in \QQ_{+}.
\end{equation*}
By multiplying by a factor $\alpha$ we can obtain that $\alpha c$, $\alpha d$ are integers and
that $K^{\vee}(\alpha)$ is an even integral lattice. Rescaling a lattice leaves its property of having
an approximate model invariant.

Finally we have the embedding
\begin{equation*}
U(\alpha c) + U(\alpha d) + K^{\vee}(\alpha) \subset U + U + K^{\vee}(\alpha)
\end{equation*}
and we can conclude from Lemma \ref{sublattice_approx_model} that $L$ has an approximate model.
\end{proof}

Lemma \ref{sublattice_approx_model} provides an approximate model for lattices of the form $U + U + L$
with $L$ integral even. The lattices that we are going to consider are not necessarily even nor admit
a decomposition $U+U+L$ but we can find an approximate model for them.

The above existence theorem is not necessarily optimal in the sense that the obtained oracle function
may get us a large numbers of possible solutions.
In our application we are in the fortunate situation that the lattice $U + U(2) + E_8(-2)$
can be trivially embedded into $U + U + E_8(-2)$ by our previous remark and so no additional work is needed.
For finding the isotropic vectors we use the algorithm of \cite{simon} implemented in \cite{pari}.

\subsection{Solution of the problems}\label{subsection_solutionsproblems}

We now use approximate models to solve the equivalence/automorphism and
representative problems that we explained at the beginning of this section.
The solutions that we provide are effective in the sense that they can be
computed on computers, but we do not make any claim on complexity, though
runtime is clearly one of our priorities.

For a lattice $L$ of signature $(p,q)$ we define $s(L) = \min(p,q)$.
For an integral lattice $L$ a {\em splitting integer} is a $\beta\in \ZZ \setminus \{0\}$
such that there exists
a vector $v$ of norm $\beta$ with $v^{\perp}$ a lattice satisfying $s(v^{\perp}) = s(L) - 1$.
Clearly, such a number exists if $s(L) \geq 1$.
We also define $r(L) = \max(p,q)$.

\begin{theorem}\label{Solution_Aut_Equi}
There exist algorithms solving the equivalence and automorphism
group problems for integral non-degenerate lattices with $r(L) \geq 5$.
\end{theorem}
\begin{proof}
The solution to those problems depends on each other, which is why they are stated together.
\begin{itemize}
\item {\bf Orth(s)}: The problem of determining a generating set of automorphism groups
 for non-degenerate lattices $L$ with $s(L) = s$.
\item {\bf Equi(s)}: Given two non-degenerate lattices $L_1$, $L_2$ with $s(L_1) = s(L_2) = s$
test whether they are isomorphic and if isomorphic find an isomorphism.
\end{itemize}
For $s(L) = 0$ or $1$ Subsection \ref{subsection_posdef_hyperbolic} provides
algorithms. Our solution is inductive in $s$. Since in the sequel we will have $s\geq 2$,
the condition of dimension at least $7$ required by Theorem \ref{Existence_ApproximateModel} is
satisfied.

If we can solve {\bf Orth(s-1)} and {\bf Equi(s-1)} then we can solve {\bf Orth(s)}.
To see this, let us take a lattice $L$ with $s(L) = s$ and $\beta$ a splitting integer.
Let us choose an approximate model $\Ap(L)$ of $L$. The oracle function will provide a set of vectors
$\Ap(L, \beta) = \{v_1, \dots, v_m\}$.
The lattice $v_1^{\perp}$ has $s(v_1^{\perp}) = s-1$. Therefore, by {\bf Orth(s-1)}
we can find $\Orth(v_1^{\perp})$.
For $v\in L$ define the sublattice  $L_v = v^{\perp} + \ZZ v$ of $L$.
The group $\Orth(v_1^{\perp})$ is naturally embedded as a subgroup $G$ of $\Orth(L_{v_1})$ by sending $v_1$ to $v_1$.
We want to determine the subgroup $H$ of $G$ that preserves $L$.
Since $L_{v_1}$ is a finite index sublattice of $L$ this can be done by applying {\bf Alg 1}.
Now we need to determine which transformations could map $v_1$ to one of $v_2$, \dots, $v_m$.
If $v_1$ is equivalent to some $v_i$ then $v_1^{\perp}$ is equivalent to $v_i^{\perp}$.
This can be tested using {\bf Equi(s-1)}. We get a corresponding map $\phi$ from $L_{v_1}$ to $L_{v_i}$.
Then by applying {\bf Alg 2} to $G$ and $\phi$ we can test whether there exists a map from $L$
to $L$ mapping $v_1$ to $v_i$. By taking those transformations when they exist and a generating
set of $H$ we  actually find a generating set of $\Orth(L)$.

If we can solve {\bf Orth(s-1)} and {\bf Equi(s-1)} then we can solve {\bf Equi(s)}.
Let us take two lattices $L$ and $L'$ with $s(L) = s(L') = s$ and $\beta$
a splitting integer of $L$. We can assume $\beta$ is a splitting integer of $L'$ since otherwise they are not equivalent.
Take a vector $v$ of norm $\beta$ in $L$ and an approximate list
$\{v'_1, \dots, v'_m\}$ of representatives in $L'$.
We compute the automorphism group $\Orth(v^{\perp})$ using {\bf Orth(s-1)} and then
the corresponding subgroup $G$ of $\Orth(L_v)$.
We simply iterate over the $v'_i$, form the lattices $v^{\perp}$ and ${v'_i}^{\perp}$
and check if there is an isomorphism using {\bf Equi(s-1)}. If there is an isomorphism $h$
we extend it to an isomorphism of $L_v$ to $L'_{v'_i}$. Then we use {\bf Alg 2} with $h$ and $G$
to check if we can obtain an isomorphism of $L$ to $L'$ mapping $v$ to $v'_i$.
If at some point we find an equivalence, then we conclude that $L$ and $L'$ are equivalent.
If not then the lattices are not.

By the work done for hyperbolic lattices, we have the solution for {\bf Orth(1)} and
{\bf Equi(1)}. Therefore,  we have the solution of {\bf Orth(s)} and {\bf Equi(s)}
for any $s\geq 2$.
\end{proof}

We next show that the assumption that $L$ be non-degenerate is actually not necessary.

\begin{theorem}\label{Solution_Aut_Equi_degenerate}
There exist algorithms for solving the equivalence and the automorphism problems for
integral lattices with $r(L) \geq 5$.
\end{theorem}
\begin{proof}
If we equip $\ZZ^n$ with a degenerate quadratic form $A$, then we can still compute the
automorphism group of this lattice. To see this, we first notice that the integral kernel $\ker(A)$ has to be preserved.
The group $\GL(\ker(A))$ is isomorphic to $\GL_k(\ZZ)$ with $k = \dim\ker(A)$.
We can always find a submodule $L'$ of $\ZZ^n$ such that $A$ restricted to $L'$ is non-degenerate and
$\ZZ^n = \ker(A) + L'$. We compute the automorphism group of $A$ restricted to $L'$
by using Theorem \ref{Solution_Aut_Equi}. Then the group $\Orth(L)$ is isomorphic to
\begin{equation*}
\GL_k(\ZZ) \rtimes \{(\ZZ^{n-k})^k \rtimes \Orth(L')\}
\end{equation*}
and so we can easily get a generating set of that group.
This method also works for isomorphism checks.
\end{proof}

\begin{lemma}\label{lemma_isotropic}
If $L$ is a lattice and ${v}$ a non-zero isotropic vector in $L$ then any automorphism of ${v}^{\perp}$
extends uniquely to an automorphism of $L\otimes \QQ$.
\end{lemma}
\begin{proof}
If $L$ is of dimension $n$ then $H = {v}^{\perp}$ is $(n-1)$-dimensional. Let $g$ be an isometry of $H$.
We want to extend this to an isometry of $L\otimes \QQ$.
We select a vector $u$ not in $H$ which gives the condition
\begin{equation*}
x.g(w) = u.w \mbox{~~for~~} w\in H\mbox{~~and~~} x = g(u).
\end{equation*}
This is an affine system for the unknown $x$. The kernel corresponds to the vectors orthogonal to $g(w)$ for $w\in H$.
Since $g$ is an automorphism of $H$ this means that the kernel is $H^{\perp} = \QQ {v}$.
Let us take a basis $h_1$, \dots, $h_{n-1}$ of $H$. The system becomes equivalent to
\begin{equation*}
x.g(h_i) = u.h_i \mbox{~~for~~} 1\leq i\leq n-1.
\end{equation*}
Since the linear system has $n$ unknowns and $n-1$ equations a solution $x = u'$ exists by the rank theorem.
Since $(u, h_1, \dots, h_{n-1})$ is of full rank, $(u', g(h_1), \dots, g(h_{n-1}))$ is also of full rank
and thus $u'\notin H$.

Thus we can write $g(u) = u' + C {v}$ for some $C\in \QQ$.
The equation $g(u).g(u) = u.u$ is expressed as $u.u = u'.u' + 2 C {v}.u'$.
We have $u'.{v} \not= 0$ because $u'\notin H$. Thus a unique solution $C$ exists.
\end{proof}

\begin{theorem}\label{subsection_beta_norm_vectors}
There exists an algorithm for computing orbit representatives of vectors of given
norm $\beta\in \RR^*$ for integral non-degenerate lattices with $r(L) \geq 6$ and $s(L) \geq 2$.
For $\beta = 0$ the algorithm gives the orbit representatives of primitive vectors.
\end{theorem}
\begin{proof}
Let us take a lattice $L$ of dimension $n$ with $s(L) = s \geq 2$.
We first use an approximate model of $L$ in order to compute an approximate
list of representatives $\{v_1, \dots, v_m\}$. The orthogonal lattice
$v^{\perp}$ satisfies $r(v^{\perp})\geq 5$ and so we can apply Theorem
\ref{Solution_Aut_Equi} to the class of lattices $v_i^{\perp}$.

If $\beta\not= 0$ then the strategy of Theorem \ref{Solution_Aut_Equi} works
to test equivalence and so reduces the approximate list to an exact list.

If $\beta = 0$ then $v_1^{\perp}$ is a lattice of dimension $n-1$ that contains
$v_1$. Thus the lattice $v_1^{\perp}$ is degenerate.
By using Theorem \ref{Solution_Aut_Equi_degenerate} we can test for isomorphisms
among the lattices $v_i^{\perp}$.
By Lemma \ref{lemma_isotropic} those isomorphisms can be lifted to isomorphisms
of the associated $\QQ$-vector spaces, and by {\bf Alg 2} we can actually
check if an integral isomorphism can be obtained.
In this way we can decide which of the $v_i$ are isomorphic.
\end{proof}

In order to compute the Tits building we must also deal with isotropic planes. For this reason we now
turn more generally to higher-dimensional isotropic $k$-planes where the situation is considerably more
complicated.
\begin{theorem}\label{theorem_isotropic_k_planes}
Let $k \geq 1$ be an integer and $L$ an indefinite non-degenerate lattice.

(i) Given an isotropic $k$-plane $Is$, we can compute the stabilizer $\Stab(L,Is)$ of $Is$ in the isometry group $\Orth(L)$
of $L$. We can also compute a finite set $(g_i)_{1\leq i\leq m}$ of elements of $\Orth(\Isperp)$ such that
\begin{equation*}
\Orth(\Isperp) = \cup_{i=1}^m g_i \Stab(L,Is)_{\Isperp}
\end{equation*}
with $\Stab(L,Is)_{\Isperp}$ the restriction of $\Stab(L,Is)$ to $\IsperpB$.

(ii) Given two isotropic $k$-planes $Is_1$ and $Is_2$, we can test whether there is an isometry of $L$ mapping $Is_1$ to $Is_2$.
\end{theorem}
\begin{proof}
(i) Let us take a basis $(e_{k+1}, \dots, e_{2k})$ of $Is$. We have $Is \subset \Isperp$ and so we can
complete this to a basis $(e_{k+1}, \dots, e_n)$ of $\IsperpB$.
We then complete this to a basis of $L$ by finding suitable vectors $(e_1, \dots, e_k)$.

The matrix of scalar products is expressed in this basis as
\begin{equation*}
B = \left(\begin{array}{ccc}
  H     &  J  &  K\\
  J^T   &  0  &  0\\
  K^T   &  0  &  A
\end{array}\right)
\end{equation*}
with $J$ a non-degenerate $k\times k$-matrix and $A$ a non-degenerate symmetric $(n-2k) \times (n-2k)$-matrix.
The matrix of scalar product of $\Isperp$ in the basis $(e_{k+1}, \dots, e_n)$ is
\begin{equation*}
C = \left(\begin{array}{cc}
0   & 0\\
0   & A
\end{array}\right).
\end{equation*}
Let us take an isometry $Q$ of $\IsperpB$. It will preserve $Is$ and its expression in $(e_{k+1}, \dots, e_n)$ is
\begin{equation*}
Q = \left(\begin{array}{cc}
Q_1   & 0\\
Q_2   & Q_3
\end{array}\right)
\end{equation*}
with $Q_3 A Q_3^T = A$. Here we recall that we use the action on row vectors from the right.

If the isometry $Q$ has an extension $P$ to $L\otimes \QQ$
then this extension satisfies $PBP^T = B$ and will necessarily
be of the form
\begin{equation*}
P = \left(\begin{array}{ccc}
  P_1 & P_2 & P_3\\
  0   & Q_1 & 0\\
  0   & Q_2 & Q_3
\end{array}\right).
\end{equation*}
When expanding the expression $P B P^T = B$ we obtain the equations
\begin{equation*}
\begin{array}{rcl}
H &=& P_1 H P_1^T + \{P_2 J^T P_1^T + P_1 J P_2^T\} + \{P_3 K^T P_1^T + P_1 K P_3^T\} + P_3 A P_3^T\\
J &=& P_1 J Q_1^T\\
K &=& P_1 J Q_2^T + P_1 K Q_3^T + P_3 A Q_3^T.
\end{array}
\end{equation*}
The second equation determines $P_1\in \GL_k(\QQ)$ uniquely. Then the third equation will determine $P_3\in M_{k,n-2k}(\QQ)$ uniquely.
However, the first equation will leave $P_2$ underdetermined which is a major complication in the case $k > 1$.

Let us take $G_1 = \Orth(\Isperp)$. We have $J^T = Q_1 J^T P_1$ which implies
\begin{equation*}
P_1^{-1} = (J^T)^{-1} Q_1 J^T.
\end{equation*}
This implies in turn that if we force $Q_1$ to preserve the lattice $L_{J^T}$ spanned by the rows of the matrix $J^T$,
then $P_1$ is integral.
By applying a conjugacy transformation and back we can apply {\bf Alg 1} to the lattice $L_{J^T}$ instead of $\ZZ^k$. So, we obtain a finite index subgroup $G_2$ of $G_1$ that preserves $L_{J^T}$.
Also using {\bf Alg 3} we can obtain the cosets of $G_2$ in $G_1$.

The equation
\begin{equation*}
  P_3 = ( K - P_1 J Q_2^T - P_1 K Q_3^T ) (Q_3^T)^{-1} A^{-1}
\end{equation*}
implies that there exist a denominator $d_3$ such that $P_3 \in \frac{1}{d_3} M_{k,n-2k}(\ZZ)$, for example $d_3 = \left\vert \det(A)\right \vert$.
The equation for $P_2$ that we obtain is
\begin{equation}\label{Equa_P2}
(P_1 J P_2^T)^T + P_1 J P_2^T = H - P_1 H P_1^T - P_3 A P_3^T - \{P_3 K^T P_1^T + P_1 K P_3^T\}
\end{equation}
We interpret this as a system of linear equations for $P_2$. Since $P_1$ and $J$ are non-degenerate we can equivalently interpret this as linear for $P_1 J P_2^T$.
The right hand side of this system of equations is symmetric. Since any equation of the form $X^T+X=M$ with $M$ symmetric obviously has a solution, e.g.  $X=M/2$,
it follows that Equation (\ref{Equa_P2}) has a solution $P_2$.
The kernel of this linear system has dimension $k(k-1)/2$.
We can find a denominator $d_2$ such that for any $Q\in G_1$
there exists a solution $P_2$ in $\frac{1}{d_2} M_{k,k}(\ZZ)$.
To be more precise a possible denominator of the right hand side of Equation \eqref{Equa_P2} is $d_3^2$.
So, a possible denominator of $P_1 J P_2^T$ is $2d_3^2$ and so a denominator of $P_2$ is $2d_3^3$.
Define $d$ as the lowest common multiple of $d_2$ and $d_3$. We define the sublattice
\begin{equation*}
L_3 = \ZZ e_1 + \dots + \ZZ e_k + \ZZ d e_{k+1} + \dots + \ZZ d e_n \subset L.
\end{equation*}
Any solution of Equation \eqref{Equa_P2} in $\frac{1}{d} M_{k,k}(\ZZ)$ will preserve $L_3$.

We define the group $H_2$ of matrices $P\in \GL_n(\QQ)$ which preserve $L_3$ and $\Isperp$ and whose restriction
to $\Isperp$ belongs to $G_2$. Thus the natural mapping $\phi:H_2\rightarrow G_2$ is surjective. By applying {\bf Alg 1}
we can get a finite index subgroup $H_3\subset \GL_n(\ZZ)$ of $H_2$. The group $H_3$ is the group $\Stab(L,Is)$, that
is the group of isometric transformation of $L$ preserving $Is$.

By applying {\bf Alg 3} we can obtain a coset decomposition of $H_3$ in $H_2$. We also have a coset decomposition of $G_2$ in $G_1$:
\begin{equation*}
H_2 = \cup_{u\in U} u H_3, \mbox{  } G_1 = \cup_{v\in V} v G_2 \mbox{ with } U \subset H_3, V\subset G_1 \mbox{ and } U,V \mbox{ finite}.
\end{equation*}
By applying $\phi$ to the first decomposition and substituting we obtain
\begin{equation*}
G_1 = \cup_{u\in U, v\in V} v \phi(u) G_3
\end{equation*}
which is the required finite coset covering. It is only a covering and not a decomposition since some of the cosets may coincide.

(ii) The process works similarly. We compute the equivalence for the spaces $L_{J_1^T}$ and $L_{J_2^T}$.
If they are not equivalent then the spaces are not equivalent. Otherwise we map the equivalence, build the corresponding spaces
and then use {\bf Alg 2} to conclude.
\end{proof}

The algorithm used in this construction is relatively complex. It would have been simpler if we had a sublattice $L'$ of $L$
such that for any $f\in O(\Isperp)$ there exists an extension that preserves $L'$. Unfortunately, we could not find a universal construction 
of such a lattice. However, in all the cases we considered, a practical algorithm allowed us to solve this problem. 

In Theorem \ref{subsection_beta_norm_vectors} we established an algorithm for computing isotropic lines. We shall now extend this to arbitrary dimension.
\begin{theorem}
There exists an algorithm for computing the orbits of isotropic  $k$-planes
of indefinite lattices $L$.
\end{theorem}
\begin{proof}
The algorithm is constructed by induction on the dimension $k$ of the isotropic spaces starting with  $k=1$, which is Theorem \ref{subsection_beta_norm_vectors}.
Suppose we know some orbit representatives of isotropic $k-1$-dimensional planes.
For each such representative $Is$, we compute the lattice $\Isperp$ which we
decompose as a lattice sum $Is + K$. This is actually also an orthogonal decomposition since $K\subset \IsperpB$.
We enumerate the orbits of isotropic primitive vectors in
$K$ for the group $\Orth(K)$ using Theorem \ref{subsection_beta_norm_vectors}
and obtain some representatives $v_1$, \dots, $v_l$.
Those can also be interpreted as isotropic $k$-planes $Is+\ZZ v_i$ in $Is+K$
for the group $\Orth(Is+K)$.

By using Theorem \ref{theorem_isotropic_k_planes} (i) we can compute the
stabilizer $\Stab(L,Is)$ of $\Isperp$ in $L$. We can further compute a covering of the cosets of $\Stab(L,Is)$
restricted to $Is+K$ in $\Orth(Is+K)$. If the cosets are $g_1$, \dots, $g_m$
then this gets us candidates $g_j ( Is + \ZZ v_i )$ for the isotropic $k$-planes
containing $Is$ covering all orbits.

We then apply Theorem \ref{theorem_isotropic_k_planes} (ii) to compute a complete list of mutually non-equivalent isotropic $k$-planes.
\end{proof}

We also note that the algorithm can be extended to enumerating flags of isotropic spaces. We simply need to
replace the group $\GL_{\dim \ker(A)}(\ZZ)$ in Theorem \ref{Solution_Aut_Equi_degenerate}
by the integral stabilizer of the flag which is isomorphic to a group of invertible
triangular matrices.

\subsection{Relationship with work by Dawes}\label{seubsec:Dawes}
Dawes \cite{Da} also developed algorithms for orthogonal groups, in particular the computation of the Tits buildings. His work is not concerned with
moduli problems of polarized Enriques surfaces, which were the starting point of our investigations.
Here we want to comment on similarities and differences in our approaches. Some of Dawes' techniques are similar to ours.
His Algorithms 2.1 and 2.2 use the same strategy as the one we implemented. However,
Dawes does not have our integral group algorithms and so he is forced to iterate
over group elements, which can be expensive.
Instead, the author uses an alternative approach: he uses the fact that some genera are known to
have only one class (see Theorem 2.3) which allows him to prove some isomorphisms relatively easily.
However, genus theory, while computationally much easier, does not provide explicit isomorphisms and
does not give a generating set of the automorphism group of a lattice. Another idea used in \cite{Da}
is to use Vinberg's algorithm. This can be done, provided the lattice is reflective, which is clearly
a substantial restriction.
In Algorithm 3.1 Dawes' approach seems needlessly complicated, since he does not use
the notion of double coset, which is exactly what one needs when splitting orbits.
This forces him to use iteration over group elements to find the matching cosets.

\vspace*{0.7cm}
\noindent
\begin{minipage}{0.5\textwidth}
Mathieu Dutour Sikiri\'c

\noindent
Rudjer Boskovi\'c Institute,

\noindent
Bijenicka 54,

\noindent
10000 Zagreb

\noindent
Croatia

\noindent
{\tt mathieu.dutour@gmail.com}
\end{minipage}

\vspace*{0.7cm}
\noindent
\begin{minipage}{0.5\textwidth}
Klaus Hulek

\noindent
Institut f\"ur Algebraische Geometrie

\noindent
Leibniz Universit\"at Hannover

\noindent
D-30060 Hannover

\noindent
Germany

\noindent
{\tt hulek@math.uni-hannover.de}
\end{minipage}
\hfill

\end{document}